\documentclass[psamsfonts,a4paper]{amsart}

\usepackage{graphicx}
\usepackage{pst-all}
\usepackage{amsmath}
\usepackage{amsfonts}
\usepackage{amssymb}
\usepackage{url}
\usepackage{fancybox}
\usepackage{subfigure}
\usepackage{nicefrac}
\input xy
\xyoption{all}

\newtheorem{theorem}{Theorem}

\newtheorem{definition}[theorem]{Definition}
\newtheorem{example}[theorem]{Example}
\newtheorem{lemma}[theorem]{Lemma}
\newtheorem{proposition}[theorem]{Proposition}
\newtheorem{remark}[theorem]{Remark}

\begin{document}

\author{J. J. S\'anchez-Gabites}
\thanks{The author is partially supported by MICINN (grant MTM 2009--07030).}
\address{Facultad de Ciencias Econ{\'{o}}micas y Empresariales. Universidad Aut{\'{o}}noma de Madrid, Campus Universitario de Cantoblanco. 28049 Madrid (Espa{\~{n}}a)}
\email{JaimeJ.Sanchez@uam.es}

\title[Arcs, balls and spheres that cannot be attractors in $\mathbb{R}^3$]{Arcs, balls and spheres that cannot be attractors in $\mathbb{R}^3$}
\begin{abstract} For any compact set $K \subseteq \mathbb{R}^3$ we define a number $r(K)$ that is either a nonnegative integer or $\infty$. Intuitively, $r(K)$ provides some information on how wildly $K$ sits in $\mathbb{R}^3$. We show that attractors for discrete or continuous dynamical systems have finite $r$ and then prove that certain arcs, balls and spheres cannot be attractors by showing that their $r$ is infinite.
\end{abstract}
\thanks{The author wishes to express his deepest gratitude to Professor Rafael Ortega (Universidad de Granada) for his generous support and encouragement during the writing of this paper}
\maketitle

\section{Introduction}

Many dynamical systems have an \emph{attractor}, which is a compact invariant set $K$ such that every initial condition sufficiently close to $K$ approaches it as the system evolves (see Definition \ref{def:atrac}). Eventually the dynamics of the system will be indistinguishable from the dynamics on the attractor, and in this sense we may say that the attractor captures the long term behaviour of the system (at least for a certain range of initial conditions).
\smallskip

The structure of attractors as subsets of the phase space $M$ is frequently very intrincate, and this leads naturally to the following \emph{characterization problem}:

``characterize \emph{topologically} what compact sets $K \subseteq M$ can be an attractor for a dynamical system on $M$''.

Notice that no conditions are placed on the dynamics on the attractor: this is a very interesting variant of the problem (for instance, it arises in some questions related to partial differential equations; see Section \ref{sec:final}.1), but almost intractable by our present knowledge. Also, the characterization will generally depend on $M$ and whether one is interested in discrete or continuous dynamical systems. The latter are well understood \cite{gunthersegal1, mio3} so we shall restrict ourselves to discrete dynamical systems.
\smallskip

Heuristically, solving the characterization problem requires that we identify the obstructions that prevent a set from being an attractor. For example, attractors have finitely generated \v{C}ech cohomology with $\mathbb{Z}_2$ coefficients \cite{mio4}, which proves that the sets shown in Figure \ref{fig:fig2} cannot be attractors. We could say that this is an \emph{intrinsic} obstruction, in the sense that it only depends on the set $K$.

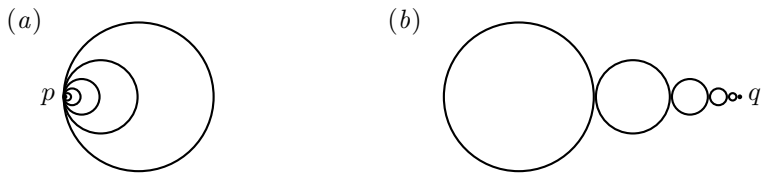
\begin{figure}[h!]
\begin{pspicture}(-1.5,-0.25)(8.2,2.2)
\pscircle(-0.96875,1){0.03125} \pscircle(-0.9375,1){0.0625} \pscircle(-0.875,1){0.125} \pscircle(-0.75,1){0.25} \pscircle(-0.5,1){0.5} \pscircle(0,1){1}
\rput[r](-1.1,1){$p$} \rput(8.1,1){$q$}
\rput(3.5,2){({\it b}\/)} \rput(-1.5,2){({\it a}\/)}
\pscircle(5,1){1} \pscircle(6.5,1){0.5} \pscircle(7.25,1){0.25} \pscircle(7.625,1){0.125} \pscircle(7.8125,1){0.0625} \pscircle(7.90625,1){0.03125}
\end{pspicture}
\caption{Sets that cannot be attractors}
\label{fig:fig2}
\end{figure}

In fact, when $M = \mathbb{R}^2$, the obstruction just mentioned is the only one. If a compact set $K \subseteq \mathbb{R}^2$ has finitely generated \v{C}ech cohomology with $\mathbb{Z}_2$ coefficients then an argument of G\"unther and Segal \cite[Corollary 3, p. 326]{gunthersegal1} shows that it is an attractor for a suitable flow, and consequently also for a homeo{\-}morphism (namely, the time-one map of the flow). This solves the characterization problem for $M = \mathbb{R}^2$: a compact set $K \subseteq \mathbb{R}^2$ is an attractor for a homeomorphism if, and only if, it has finitely generated \v{C}ech cohomology with $\mathbb{Z}_2$ coefficients.
\smallskip

Let us move on to $M = \mathbb{R}^3$, which will be our framework from now on. With the increase in dimension a new phenomenon arises, related to the fact that attractors need to \emph{sit} in phase space in a suitable way. For instance, consider a straight segment $A$ joining two points in $\mathbb{R}^3$. It is easy to prove that $A$ is an attractor for a homeomorphism (Example \ref{ex:segment}), but we will show in Theorem \ref{teo:wild_arc} that it is possible to embed $A$ in some wild fashion so that the resulting arc $A^*$ \emph{cannot} be an attractor any more (a picture of $A^*$ can be seen Figure \ref{fig:wild_1}). Note that $A^*$ is homeomorphic to $A$, so it still has the ``correct'' \emph{intrinsic} topological properties to be an attractor, but the way it sits in $\mathbb{R}^3$ prevents this from happening. In this sense we may say that there are also \emph{extrinsic} obstructions involved in the characterization problem in $\mathbb{R}^3$, standing in contrast with the case of $\mathbb{R}^2$.
\smallskip

Our aim in this paper is not to solve the characterization problem (we are far from being able to do this) but to gain some understanding about it by exploring the phenomenon just described, showing that there are not only arcs $A^*$ that cannot be attractors, but also balls $B^*$ and spheres $S^*$ (Theorems \ref{teo:ball} and \ref{teo:alex}). The interesting point about these results is the technique used in proving them, which can be roughly described as follows. We associate to every compact set $K \subseteq \mathbb{R}^3$ a number $r(K) \in \{0,1,2,\ldots,\infty\}$ that, on intuitive grounds, measures the ``crookedness'' of $K$ as a subset of $\mathbb{R}^3$ (Section \ref{sec:motiva}). Attractors have finite $r$, so we only need to show that there are arcs, balls, and spheres with $r = \infty$. However, $r$ is notoriously difficult to compute, and some powerful tools (the nullity and subadditivity properties of $r$) will be needed in order to complete our plan (Sections \ref{sec:fin} to \ref{sec:sub}).

\subsection*{Requirements} Care has been taken in making the paper as self contained as possible, although some knowledge of singular or simplicial homology and cohomology is required, including the duality theorems of Alexander and Lefschetz. The books by Hatcher \cite{hatcher1} and Munkres \cite{munkres3} contain everything that is needed for our purposes. Throughout the paper we tacitly use $\mathbb{Z}_2$ coefficients as a convenient choice to make some arguments simpler and be able to switch between homology and cohomology, but any other coefficient field would work just as well.

For an abelian group $G$ the notation ${\rm rk}\ G$ means the rank of $G$. In our context $G$ will usually be a homology or cohomology group with $\mathbb{Z}_2$ coefficients (hence a vector space over the field $\mathbb{Z}_2$) and then ${\rm rk}\ G$ agrees with the perhaps more familiar dimension of $G$ as a vector space.

\section{An arc that cannot be an attractor} \label{sec:arcs}

In order to provide the reader with some familiarity with attractors and the computation of $r$ we prove in this section that there exist arcs that cannot be attractors. We begin with some basic definitions.
\smallskip

Let $U \subseteq \mathbb{R}^3$ be open. A \emph{discrete dynamical system} on $U$ is an injective and continuous mapping $f : U \longrightarrow U$. A compact set $K \subseteq U$ is called \emph{invariant} if $f(K) = K$. A compact set $P \subseteq U$ is said to be \emph{attracted} by $K$ if for every neighbourhood $V$ of $K$ there exists $n_0 \in \mathbb{N}$ such that $f^n(P) \subset V$ whenever $n \geq n_0$.

\begin{definition} \label{def:atrac} Let $f : U \longrightarrow U$ be a discrete dynamical system on an open set $U \subseteq \mathbb{R}^3$. A set $K \subseteq U$ is an \emph{attractor} for $f$ if:
\begin{itemize}
	\item[({\it i}\/)] $K$ is compact and invariant,
	\item[({\it ii}\/)] $K$ has a neighbourhood $A \subseteq U$ such that every compact set $P \subseteq A$ is attracted by $K$.
\end{itemize}
\end{definition}

The biggest set $A$ for which condition ({\it ii}\/) is satisfied is called the \emph{basin of attraction} of $K$ and denoted $\mathcal{A}(K)$. It is always an open invariant subset of $U$ and can be alternatively characterized as \[\mathcal{A}(K) = \{p \in U : f^n(p) \longrightarrow K \text{ as } n \longrightarrow \infty\},\] where the condition $f^n(p) \longrightarrow K$ should be understood in the sense that $f^n(p)$ eventually enters (and never leaves again) any prescribed neighbourhood of $K$. An attractor is called \emph{global} if its basin of attraction coincides with all of $U$.
\smallskip

Definition \ref{def:atrac} is pretty standard \cite{hale1,katok1}, although some authors consider only global attractors or minimal attractors. Both restrictions are unnecesary in our context.

\begin{example} \label{ex:trivial} The unit sphere $\mathbb{S}^2 \subseteq \mathbb{R}^3$ is an attractor for the homeomorphism $f : \mathbb{R}^3 \longrightarrow \mathbb{R}^3$ defined as $f(p) := p \sqrt{\|p\|}$, its basin of attraction being $\mathcal{A}(\mathbb{S}^2) = \mathbb{R}^3 - \{0\}$. A very similar construction proves that the closed unit ball $\mathbb{B}^3 \subseteq \mathbb{R}^3$ is a global attractor.
\end{example}

\begin{example} \label{ex:segment} Any straight line segment $L \subseteq \mathbb{R}^3$ is an attractor. To prove this, first let $h : \mathbb{R}^3 \longrightarrow \mathbb{R}^3$ be a homeomorphism that takes $L$ onto the straight line segment $L'$ having endpoints $(-1,0,0)$ and $(1,0,0)$. Consider the homeomorphism $\varphi : \mathbb{R} \longrightarrow \mathbb{R}$ given by \[\varphi(x) := \left\{ \begin{array}{ccc} {\rm sgn}(x)\sqrt{|x|} & \text{ if } & |x| \geq 1 \\ x & \text{ if } & |x| \leq 1 \end{array} \right.\] It is easy to see that $[-1,1]$ is a global attractor for $\varphi$. Then the homeomorphism $f' : \mathbb{R}^3 \longrightarrow \mathbb{R}^3$ defined by $f'(x,y,z) := (\varphi(x),\nicefrac{y}{2},\nicefrac{z}{2})$ has $L'$ as a global attractor, and $f := h^{-1} \circ f' \circ h$ has $L$ as a global attractor.
\end{example}

A \emph{ball} is a set homeomorphic to $\mathbb{B}^3$, a \emph{sphere} is a set homeomorphic to $\mathbb{S}^2$, and an \emph{arc} is a set homeomorphic to a straight line segment in $\mathbb{R}^3$. Examples \ref{ex:trivial} and \ref{ex:segment} imply that, whatever intrinsic obstructions there may be for a set to be an attractor, they are not present for balls, spheres, and arcs.
\smallskip

As mentioned earlier, our goal in this section is to construct an arc that cannot be an attractor. We shall approach this by first constructing an arc that cannot be a \emph{global} attractor, which is a much easier task. The concept of cellularity \cite[p. 35]{daverman1} will play an important role in the sequel:

\begin{definition} A compact set $K \subseteq \mathbb{R}^3$ is \emph{cellular} if it has a neighbourhood basis comprised of balls.
\end{definition}

Consider the arc $F$ shown in Figure \ref{fig:wild_2}. This arc was introduced in a classical paper by Fox and Artin \cite[Example 1.3 and Figure 8, p. 985]{foxartin1}, where it is also shown that $F$ is not cellular.

\begin{figure}[h]
\begin{pspicture}(0,0)(8.5,2)
%\psgrid(0,0)(8.5,2)
\rput[bl](0,0){\scalebox{0.5}{\includegraphics{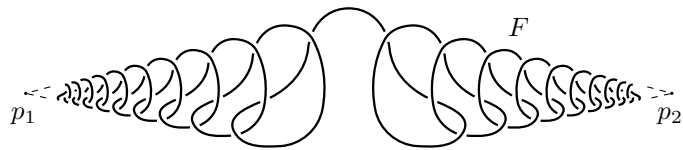}}}
\rput[t](0,0.8){$p_1$} \rput[t](8.5,0.8){$p_2$}
\rput[b](6.5,1.75){$F$}
\end{pspicture}
\caption{The Fox--Artin arc $F$}
\label{fig:wild_2}
\end{figure}

\begin{example} \label{ejem:toy} The arc $F$ shown in Figure \ref{fig:wild_2} cannot be a \emph{global} attractor in $\mathbb{R}^3$.
\end{example}
\begin{proof} We argue by contradiction. Suppose $F$ is an attractor for a continuous and injective map $f : \mathbb{R}^3 \longrightarrow \mathbb{R}^3$. Let $B$ be a ball big enough so that $F$ is contained in its interior and set $N_k := f^k(B)$ for $k \in \mathbb{N}$.

The theorem on invariance of domain \cite[Theorem 36.5, p. 207]{munkres3} guarantees that $f$ is an open map, which implies that each $N_k$ is a neighbourhood of $f^k(F) = F$ (recall that an attractor is required to be invariant). Moreover, since $B$ is a compact subset of the basin of attraction of $F$ because $F$ was assumed to be a global attractor and $F$ attracts compact sets, the $N_k$ get arbitrarily close to $F$ as $k$ increases, so $\{N_k\}$ is a neighbourhood basis of $F$.

Each map $f^k$ is continuous and injective so, again because of the compactness of $B$, the restrictions $f^k|_B : B \longrightarrow f^k(B) = N_k$ are all homeomorphisms. Consequently all the $N_k$ are balls and therefore $F$ should be cellular, which it is not as mentioned earlier. This contradiction finishes the proof.
\end{proof}

This very same argument proves that a global attractor in $\mathbb{R}^3$ must be cellular. The converse is also true: a cellular subset of $\mathbb{R}^3$ is an attractor for some discrete dynamical system; in fact, it is an attractor for a flow \cite[Theorem 2.7, p. 152]{garay1}. This solves the characterization problem for \emph{global} attractors both for continuous and discrete dynamical systems.
\smallskip

Let us now elaborate on the Fox--Artin arc $F$ to construct another arc $A^*$ that cannot be an attractor whatsoever; that is, we do \emph{not} assume now that $A^*$ is a global attractor. The argument of Example \ref{ejem:toy} is no longer useful in this situation because there may not exist a ball $B$ contained in the region of attraction of $A^*$. At this point we need to turn to the invariant $r$ mentioned in the introduction. The precise definition is postponed to the following section, and for now we just enumerate some of its properties (in a slightly informal fashion) which enable us to make some computations and prove that $A^*$ cannot be an attractor.

The first one provides the link between $r(K)$ and dynamics:

\begin{itemize}
	\item the \emph{finiteness property}: if $K$ is an attractor for some dynamical system, then $r(K) < \infty$.
\end{itemize}

Thus to prove that a given compact set $K$ cannot be an attractor it is enough to show that $r(K) = \infty$. This is not easy to do directly, and we need to use some ``geometric'' properties of $r(K)$ as an aid in this task. They are not subtle enough to allow us to compute $r(K)$ exactly, but they suffice to show that $r(K) = \infty$ for certain compacta $K$. These geometric properties are the following:

\begin{itemize}
	\item the \emph{invariance property}: if $K$ and $K'$ are ambient homeomorphic, then $r(K) = r(K')$;
	\item the \emph{nullity property}: if $K$ is an arc or a ball, then $r(K) = 0$ if and only if $K$ is cellular;
	\item the \emph{subadditivity property}: under suitable hypotheses, if $K$ is expressed as the union of two compact sets $K_1$ and $K_2$, then $r(K) \geq r(K_1) + r(K_2)$.
\end{itemize}

Recall that two subsets $K$ and $K'$ of $\mathbb{R}^3$ are said to be \emph{ambient homeomorphic} if there exists a homeomorphism $h : \mathbb{R}^3 \longrightarrow \mathbb{R}^3$ such that $h(K) = K'$. This notion captures the intuitive idea that $K$ and $K'$ sit in $\mathbb{R}^3$ in equivalent ways. Ambient homeomorphic sets are also frequently said to be topologically equivalently embedded, topologically equivalent, or simply equivalent.
\smallskip

Now we can construct an arc $A^*$ and show that it cannot be an attractor.

\begin{theorem} \label{teo:wild_arc} There is an arc $A^* \subseteq \mathbb{R}^3$ that cannot be an attractor.
\end{theorem}
\begin{proof} Let $F_1 := F$, the Fox--Artin arc considered earlier, with endpoints $p_1$ and $p_2$. Place another copy of $F$, scaled down by a factor of $\frac{1}{2}$, next to $F_1$. Its left endpoint coincides with $p_2$; its right endpoint is denoted $p_3$. Denote this copy $F_2$ and continue in the same fashion adding more copies $F_3, F_4, \ldots$ as in Figure \ref{fig:wild_1}. The right endpoints of these arcs accumulate at a point $p_{\infty}$. Let $A^* := \bigcup_{n=1}^{\infty} F_n \cup p_{\infty}$. It is clear that $A^*$ is an arc with endpoints $p_1$ and $p_{\infty}$.

\begin{figure}[h]
\begin{pspicture}(0,0)(12.5,1.6)
%\psgrid(0,0)(12.5,1.6)
\rput[bl](0,0){\scalebox{0.5}{\includegraphics{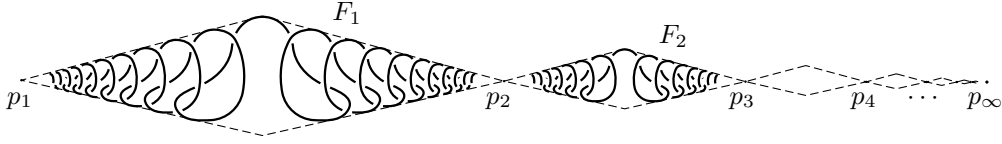}}}
\rput[t](0,0.6){$p_1$} \rput[t](6.3,0.6){$p_2$} \rput[t](9.5,0.6){$p_3$} \rput[t](11.1,0.6){$p_4$} \rput[t](11.9,0.6){$\cdots$} \rput[t](12.7,0.6){$p_{\infty}$}
\rput[b](4.3,1.5){$F_1$} \rput[lb](8.4,1.2){$F_2$}
\end{pspicture}
\caption{The arc $A^*$}
\label{fig:wild_1}
\end{figure}

By the finiteness property, in order to prove that $A^*$ cannot be an attractor it is enough to show that $r(A^*) = \infty$. Express $A^*$ as the union $A^* = F_1 \cup p_2 p_{\infty}$, where $p_2 p_{\infty}$ denotes the subarc of $A^*$ comprised between $p_2$ and $p_{\infty}$. The subadditivity property implies that $r(A^*) \geq r(F_1) + r(p_2 p_{\infty})$. Repeating the same process with $p_2 p_{\infty}, p_3p_{\infty},\ldots$ we see inductively that \[r(A^*) \geq r(F_1) + r(F_2) + \ldots + r(F_s) + r(p_{s+1}p_{\infty})\] for any $s \geq 0$. Since $r$ always takes nonnegative values, \[r(A^*) \geq r(F_1) + r(F_2) + \ldots + r(F_s).\]

Each of the arcs $F_i$ is a scaled down copy of $F$; hence they are all ambient homeomorphic to each other and therefore $r(F_i) = r(F)$ by the invariance property. We mentioned earlier that $F$ is not cellular so by the nullity property we have $r(F) \geq 1$; consequently $r(F_i) \geq 1$ for every $i$ too. Thus $r(A^*) \geq s$. Since $s$ was arbitrary, we conclude that $r(A^*) = \infty$.
\end{proof}

There is nothing special about using copies of the same arc $F$ in the construction of $A^*$ in Theorem \ref{teo:wild_arc}; it is just the simplest choice. Placing whatever arcs $A_i$ with endpoints $p_i$ and $p_{i+1}$ in each of the diamond shaped regions of Figure \ref{fig:wild_1} we see from the subadditivity property that \[r(A^*) \geq \sum_{i=1}^{\infty} r(A_i)\] and so if infinitely many of the $A_i$ are noncellular (thus they satisfy $r(A_i) \geq 1$) then $r(A^*) = \infty$, which implies that $A^*$ cannot be an attractor. It is not difficult to see that this construction can be used to produce uncountably many non ambient homeomorphic arcs (so they are all ``different'') none of which can be an attractor.

\section{The definition of $r(K)$} \label{sec:motiva}

Our definition of $r(K)$ is suggested by the argument of Example \ref{ejem:toy}. The key idea there was that if the Fox--Artin arc $F$ were an attractor, then by picking a ball $B$ that is a neighbourhood of $F$ contained in its region of attraction and iterating $B$ with the map $f$ we would obtain a neighbourhood basis of the arc comprised of balls, and we know that such a basis does not exist because the arc is not cellular. We also mentioned that in general this argument breaks down because the existence of the ball $B$ is not guaranteed if the basin of attraction of the presumed attractor is not assumed to be all of $\mathbb{R}^3$.
\smallskip

Let us now deal informally (details will be given throughout this section) with the general case, where no assumption is placed on $\mathcal{A}(K)$. If $K$ is an attractor for $f$ then, since $\mathcal{A}(K)$ is an open neighbourhood of $K$, we can find \emph{some} neighbourhood of $K$ that is a compact $3$--manifold (with boundary) $N$ contained in $\mathcal{A}(K)$. Iterating $N$ we obtain a neighbourhood basis $\{N_k\}$ of $K$ comprised of compact manifolds, \emph{all homeomorphic to each other}. So in order to find a compact set $K$ that cannot be an attractor we need to find a set that does \emph{not} have such a neighbourhood basis. For the sake of brevity we call a neighbourhood basis of $K$ whose members are compact manifolds an $m$--neighbourhood basis.
\smallskip

The simplest way to tell that not all the members of an $m$--neighbourhood basis $\{N_k\}$ are homeomorphic is by comparing their homology groups. For instance, we may compare the ranks of their first homology groups with coefficients in $\mathbb{Z}_2$. These ranks are finite because the $N_k$ are compact manifolds. Thus, given an $m$--neighbourhood basis $\{N_k\}$ of $K$, we can consider the limit inferior \[\liminf_{k \rightarrow \infty} \ {\rm rk}\ H_1(N_k;\mathbb{Z}_2)\] and if this limit is $\infty$ then not all the $\{N_k\}$ are homeomorphic (otherwise their first homology groups would all have the same finite rank $r$ and the limit inferior would be precisely this number $r$).
\smallskip

Finally, if we want to show that $K$ cannot be an attractor then the limit inferior described above has to be infinite for every $m$--neighbourhood basis $\{N_k\}$, so we define \[r(K) := \inf\left\{ \liminf_{k \rightarrow \infty}\ {\rm rk}\ H_1(N_k;\mathbb{Z}_2) : \{N_k\} \text{ is an $m$--neighbourhood basis of } K\right\}\] and conclude that if $r(K) = \infty$, then $K$ cannot be an attractor whatsoever. The reader may recognise this as an equivalent statement of the finiteness property that was introduced in the previous section.
\smallskip

The computation of $r(K)$ is usually very difficult. Because it is an infimum of a set of numbers it is generally easy to bound from above, but hard to bound from below. In this paper we are interested in showing that certain sets have infinite $r$, which corresponds exactly to the hard lower bounds problem that we just mentioned. This is why we need the nullity and subadditivity properties (notice that both bound $r(K)$ from below). These follow from some ``woodworking'' constructions with neighbourhoods of $K$, which requires us to work with polyhedra.

\subsection{Polyhedra in $\mathbb{R}^3$ \label{sec:poliedros}} We include here a very quick review of some piecewise linear topology in $\mathbb{R}^3$, loosely following the books by Hudson \cite{hudson1} and Moise \cite{moise2} but tailoring everything to our fairly modest needs.

A \emph{$d$--dimensional simplex} ($0 \leq d \leq 3$) in $\mathbb{R}^3$ is the convex hull $\sigma$ of $d+1$ affinely independent points in $\mathbb{R}^3$, called the \emph{vertices} of $\sigma$. A \emph{proper face} of $\sigma$ is a simplex spanned by some, but not all, of the vertices of $\sigma$. According to this, a $0$--simplex is just a point (with no proper faces). A $1$--simplex is a segment, whose proper faces are its endpoints. A $2$--simplex is a triangle, whose proper faces are its vertices and edges. And a $3$--simplex is a tetrahedron, whose proper faces are its vertices, edges and faces (in the elementary geometry sense of the word).

Although very restrictive, the following definition will be enough for our purposes: a \emph{polyhedron} $P \subseteq \mathbb{R}^3$ is the union of a finite collection of simplices, not necessarily all of the same dimension. Therefore, for us, polyhedra are compact. It is important to distinguish between a polyhedron, which is a subset of $\mathbb{R}^3$, and its expression as a finite union of simplices. The latter is clearly not unique and may be very complicated. However, one can prove that any polyhedron admits especially neat expressions as finite unions of simplices, called triangulations: a finite collection of simplices $\{\sigma_i : i \in I\}$ is called a \emph{triangulation} of the polyhedron $P$ if $P = \bigcup_{i \in I} \sigma_i$ and any two different $\sigma_i$ and $\sigma_j$ are either disjoint or meet in a proper face. For instance, any two triangles ($2$--simplices) in a triangulation have to be either disjoint, meet in a single vertex or meet along a single edge.

Because of the way they are constructed, polyhedra sit nicely (piecewise linearly) in Euclidean space, making them especially suitable for geometric arguments. They also have the following useful properties, which are immediate consequences of the definition or the existence of triangulations:
\begin{itemize}
	\item finite intersections and unions of polyhedra are, again, polyhedra;
	\item polyhedra have finitely generated homology and cohomology groups.
\end{itemize}

An extremely useful tool in piecewise linear topology is the theory of regular neighbourhoods \cite[p. 57ff.]{hudson1}. If $P$ is a polyhedron in $\mathbb{R}^3$, then it has arbitrarily small neighbourhoods $N$, called \emph{regular neighbourhoods}, with the following properties: ({\it i}\/) $N$ is a polyhedron, ({\it ii}\/) $N$ is a $3$--manifold with boundary, ({\it iii}\/) the inclusion $P \subseteq N$ is a homotopy equivalence.
\medskip

The following result is very easy to prove and probably well known to the reader.

\begin{lemma} \label{lem:easy} Let $K \subseteq \mathbb{R}^3$ be compact. Then $K$ has arbitrarily small neighbourhoods $N$ that are polyhedral $3$--manifolds.
\end{lemma}

{\bf Terminology.} In the sequel we use the expressions $p$--neighbourhood, $m$--neighbourhood and $pm$--neighbourhood to denote neighbourhoods that are polyhedra, manifolds, or polyhedral manifolds respectively. The terms $p$--neighbourhood basis, $m$--neighbour{\-}hood basis and $pm$--neighbourhood basis have an analogous meaning which should be clear for the reader.

\subsection{The definition of $r(K)$} Although we started with a problem in dynamics, the definition of $r(K)$ is purely topological. From this point of view, $r(K)$ somehow measures the ``crookedness'' of the embedding of $K$ in $\mathbb{R}^3$: we approximate $K$ with nice neighbourhoods $N_k$ and keep track of how complicated these neighbourhoods have to become in order to trace $K$ ever more closely. This is the idea that we intended to highlight in the actual definition of $r(K)$ given below (Definition \ref{defn:rdek}), which is different from the one obtained above but equivalent to it, as shown at the end of this section.

For $r$ a nonnegative integer, consider the following property that $K$ may or may not have: \[(P_r) : \text{ $K$ has arbitrarily small $p$--neighbourhoods $N$ with } {\rm rk}\ H_1(N) \leq r.\]

\begin{definition} \label{defn:rdek} $r(K)$ is the smallest nonnegative integer $r$ for which $(P_r)$ holds, or $\infty$ if $(P_r)$ does not hold for any $r$.
\end{definition}

Instead of defining $r(K)$ in terms of $p$--neighbourhoods of $K$ we could have chosen $m$--neighbourhoods or even $pm$--neighbourhoods. It is a remarkable fact that all three choices lead to the same result:

\begin{proposition} \label{prop:equivalentes} Let $K \subseteq \mathbb{R}^3$ be compact. Then, for any nonnegative integer $r$ the following are equivalent:
\begin{enumerate}
	\item[({\it i}\/)] $K$ has arbitrarily small $p$--neighbourhoods $N$ with ${\rm rk}\ H_1(N) \leq r$,
	\item[({\it ii}\/)] $K$ has arbitrarily small $m$--neighbourhoods $N$ with ${\rm rk}\ H_1(N) \leq r$,
	\item[({\it iii}\/)] $K$ has arbitrarily small $pm$--neighbourhoods $N$ with ${\rm rk}\ H_1(N) \leq r$.
\end{enumerate}
Consequently $r(K)$ can be characterized as the least nonnegative integer (possibly $\infty$) for which any one of ({\it i}\/), ({\it ii}\/) or ({\it iii}\/) holds.
\end{proposition}
\begin{proof} Implications ({\it iii}\/) $\Rightarrow$ ({\it i}\/) and ({\it iii}\/) $\Rightarrow$ ({\it ii}\/) are trivial. If we prove their converses the equivalence of all three will be established.

({\it i}\/) $\Rightarrow$ ({\it iii}\/) Let $U$ be a neighbourhood of $K$ and choose a $p$--neighbourhood $N$ of $K$ contained in $U$ with ${\rm rk}\ H_1(N) \leq r$. Now let $N'$ be a regular neighbourhood of $N$ small enough to be still contained in $U$. We know that $N$ is a polyhedral manifold; moreover, since the inclusion $N \subseteq N'$ is a homotopy equivalence it follows that $H_1(N) = H_1(N')$ and therefore ${\rm rk}\ H_1(N') \leq r$. Thus we see that ({\it iii}\/) holds.

({\it ii}\/) $\Rightarrow$ ({\it iii}\/) Let $U$ be a neighbourhood of $K$ and choose an $m$--neighbourhood $N$ of $K$ contained in $U$ with ${\rm rk}\ H_1(N) \leq r$. Now we invoke a very deep theorem of Moise \cite[Theorem 1, p. 253]{moise2}:

\smallskip
{\it Theorem.} (Moise) Let $N \subseteq \mathbb{R}^3$ be a compact $3$--manifold. Let $\varepsilon > 0$ be given. There is an embedding $h : N \longrightarrow \mathbb{R}^3$ that moves points less than $\varepsilon$ (that is, the distance from $p$ to $h(p)$ is less than $\varepsilon$ for every $p \in N$) and such that $h(N)$ is a polyhedron.
\smallskip

Let $h : N \longrightarrow \mathbb{R}^3$ be the embedding provided by the theorem of Moise and let $N' := h(N)$, which is a polyhedron. Choosing $\varepsilon$ small enough guarantees that $N'$ is still a neighbourhood of $K$ contained in $U$. Since $N'$ is homeomorphic to $N$, it is a manifold and, moreover, $H_1(N') = H_1(N)$ so that ${\rm rk}\ H_1(N') \leq r$. Therefore ({\it iii}\/) holds.
\end{proof}

In the sequel we will use $p$--neighbourhoods, $m$--neighbourhoods and $pm$--neigh{\-}bour{\-}hoods interchangeably to compute $r(K)$.

\begin{example} \label{ejem:cell} If $K$ is cellular, then it has arbitrarily small neighbourhoods $B$ that are balls. Thus ({\it ii}\/) in Proposition \ref{prop:equivalentes} holds with $r = 0$, and so $r(K) = 0$.
\end{example}

To close this section we find an expression of $r(K)$ in terms of neighbourhoods bases of $K$. This will be useful later on, when we establish the subadditivity property. 

\begin{proposition} \label{prop:easy1} Let $K \subseteq \mathbb{R}^3$ be compact. Then:
\begin{enumerate}
	\item[({\it i}\/)] for every $pm$--neighbourhood basis $\{N_k\}$ of $K$ \[r(K) \leq \liminf_{k \rightarrow \infty} {\rm rk}\ H_1(N_k),\]
	\item[({\it ii}\/)] if $r(K) < \infty$, then $K$ has a $pm$--neighbourhood basis $\{N_k\}$ such that \[{\rm rk}\ H_1(N_k) = r(K)\ \ \text{for every } k.\]
\end{enumerate}

As a consequence, $r$ admits the expression \[r(K) = \inf\ \Bigl\{\liminf_{k \rightarrow \infty} {\rm rk}\ H_1(N_k) : \{N_k\} \text{ is a $pm$--neighbourhood basis of $K$}\Bigr\}.\]
\end{proposition}
\begin{proof} ({\it i}\/) Denote $r := \liminf_{k \longrightarrow \infty} {\rm rk}\ H_1(N_k)$. If $r = \infty$ there is nothing to prove, so we assume $r < \infty$. Since $r$ is the limit inferior of a sequence of nonnegative integers, there exists a subsequence $N_{k_{\ell}}$ of $N_k$ such that ${\rm rk}\ H_1(N_{k_\ell}) = r$ for every $\ell$. Now, clearly $\{N_{k_{\ell}}\}$ is still a $pm$--neighbourhood basis of $K$, and so $K$ has property $(P_r)$. Therefore $r(K) \leq r$.
\smallskip

({\it ii}\/) Let $r = r(K)$. $K$ has property $(P_r)$, so from Proposition \ref{prop:equivalentes} we see that $K$ has a $pm$--neighbourhood basis $\{N_k\}$ with ${\rm rk}\ H_1(N_k) \leq r$ for every $k$. However, $K$ does \emph{not} have property $(P_{r-1})$, and consequently every sufficiently small $pm$--neighbourhood $N$ of $K$ must satisfy ${\rm rk}\ H_1(N) \geq r$. Thus for $k$ big enough we have ${\rm rk}\ H_1(N_k) = r$ and the claim is proved.
\end{proof}

\section{The finiteness property} \label{sec:fin}

The following lemma is almost trivial but important, because it contains the observation that provides the link between dynamics and the invariant $r(K)$.

\begin{lemma} \label{lem:trivial} Let $K$ be an attractor for a discrete dynamical system $f : U \longrightarrow U$, where $U$ is an open subset of $\mathbb{R}^3$. Then $K$ has an $m$--neighbourhood basis $\{N_k\}$ with all the $N_k$ homeomorphic to each other.
\end{lemma}
\begin{proof} Since $\mathcal{A}(K)$ is open in $U$ and $U$ is open in $\mathbb{R}^3$, it follows that $\mathcal{A}(K)$ is open in $\mathbb{R}^3$. By Lemma \ref{lem:easy} there exists an $m$--neighbourhood $N$ of $K$ contained in $\mathcal{A}(K)$. Since $K$ attracts compact subsets of $\mathcal{A}(K)$, it follows that $\{N_k := f^k(N) : k \in \mathbb{N}\}$ is a neighbourhood basis of $K$ in $\mathbb{R}^3$. Since $N$ is compact and $f$ is continuous and injective, each restriction $f^k|_{N} : N \longrightarrow N_k$ is a homeomorphism onto, so every $N_k$ is homeomorphic to $N$. Thus all the $N_k$ are compact $3$--manifolds homeomorphic to $N$ and therefore to each other.
\end{proof}

\begin{theorem} \label{teo:finite} (Finiteness) Let $K \subseteq \mathbb{R}^3$ be a compact subset of $\mathbb{R}^3$ and assume that it is an attractor for a dynamical system. Then $r(K) < \infty$.
\end{theorem}
\begin{proof} By Lemma \ref{lem:trivial} $K$ has an $m$--neighbourhood basis $\{N_k\}$ comprised of compact $3$--manifolds, all homeomorphic to each other. Thus their first homology groups $H_1(N_k)$ all have the same rank $r$ which is finite because compact manifolds have finitely generated homology groups \cite[Corollaries A.8 and A.9, p. 527]{hatcher1}. From Proposition \ref{prop:equivalentes} we conclude that $r(K) \leq r < \infty$.
\end{proof}

Notice that the argument of Lemma \ref{lem:trivial} cannot be used to produce a $p$--neigh{\-}bourhood basis of $K$, because even if $N$ is chosen to be polyhedral, its images $f^k(N)$ do not need to be polyhedral. This illustrates the usefulness of the equivalences established in Proposition \ref{prop:equivalentes}.

\section{The invariance property} 

The precise statement of the invariance property and its proof are as follows:

\begin{theorem} \label{teo:inv} (Invariance) Let $K, K' \subseteq \mathbb{R}^3$ be compact sets and assume that there is a homeomorphism $h : \mathbb{R}^3 \longrightarrow \mathbb{R}^3$ such that $h(K) = K'$. Then $r(K) = r(K')$.
\end{theorem}
\begin{proof} It will be enough to show that if $K$ has property $(P_r)$ for some $r$, then so does $K'$. Let $V'$ be a neighbourhood of $K'$ and set $V := h^{-1}(V')$, which is also a neighbourhood of $K$. Since $K$ has property $(P_r)$, it has an $m$--neigbourhood $N \subseteq h^{-1}(V')$ with ${\rm rk}\ H_1(N) \leq r$. Let $N' := h(N)$, which is clearly a neighbourhood of $K'$ contained in $V'$. Since $N'$ is homeomorphic to $N$, it is a manifold (hence an $m$--neighbourhood of $K'$) with ${\rm rk}\ H_1(N') = {\rm rk}\ H_1(N) \leq r$. Thus $K'$ also has property $(P_r)$.
\end{proof}

Although Theorem \ref{teo:inv} is enough for our purposes of showing that certain sets cannot be attractors, the invariance property is in fact stronger: if the complements of $K$ and $K'$ are homeomorphic, then $r(K) = r(K')$. We now prove this version of the invariance property. For technical reasons it is convenient to consider the complements of $K$ and $K'$ in $\mathbb{S}^3$, rather than $\mathbb{R}^3$:

\begin{theorem} \label{teo:invs} Let $K, K' \subseteq \mathbb{R}^3 \subseteq \mathbb{S}^3$ be compact sets and assume that there is a homeomorphism $h : \mathbb{S}^3 - K \longrightarrow \mathbb{S}^3 - K'$. Then $r(K) = r(K')$.
\end{theorem}
\begin{proof} As in the proof of Theorem \ref{teo:inv} it will be enough to show that if $K$ has property $(P_r)$ for some $r$, then so does $K'$.

Let $V'$ be an open neighbourhood of $K'$ and set $V := K \cup h^{-1}(V'-K')$, which is an open neighbourhood of $K$. Since $K$ has property $(P_r)$, it has an $m$--neighbourhood $N \subseteq V$ such that ${\rm rk}\ H_1(N) \leq r$. Denote $N' := K' \cup h(N - K)$. It is easy to check that $N'$ is an $m$--neighbourhood of $K'$ contained in $U$, and we have \[{\rm rk}\ H_1(N') = {\rm rk}\ H_1(\mathbb{S}^3 - N') = {\rm rk}\ H_1(\mathbb{S}^3 - N) = {\rm rk}\ H_1(N) \leq r,\] where the first and third equalities follow from Alexander duality and the middle one from the fact that $\mathbb{S}^3 - N$ and $\mathbb{S}^3 - N'$ are homeomorphic via $h$ (Alexander duality is recalled below, after Definition \ref{def:bola_perforada}). We conclude that $K'$ has property $(P_r)$, as was to be proved.
\end{proof}

Theorem \ref{teo:invs} can be alternatively stated by saying that \emph{$r(K)$ is a topological invariant of $\mathbb{S}^3 - K$}.

\section{The nullity property}

We now turn to the nullity property, which is the first of our results that actually proves that $r(K)$ is nonzero for some compact set $K$, laying the basis for more elaborate constructions as in Theorem \ref{teo:wild_arc}. Its statement is more general than the one given in the introduction:

\begin{theorem} \label{teo:null} (Nullity) Let $K \subseteq \mathbb{R}^3$ be a continuum with $\check{H}^2(K) = 0$. Then \[r(K) = 0 \Leftrightarrow K \text{ is cellular.}\]
\end{theorem}

The condition $\check{H}^2(K) = 0$ guarantees that $K$ has no ``spherical holes'' and it will seem very natural after Lemma \ref{lem:perfora} below. Here $\check{H}^2(K)$ denotes the second \v{C}ech cohomology group of $K$ with coefficients in $\mathbb{Z}_2$. \v{C}ech cohomology agrees with the usual (singular, or simplicial) cohomology theory on polyhedra \cite[Theorem 73.2, p. 437]{munkres3} but provides more information when applied to compacta with a ``bad'' local structure. There are several definitions of \v{C}ech cohomology, of which the readers could choose their favourite. For our purposes it will be enough to be aware of the following particular instance of the \emph{continuity property} of \v{C}ech cohomology \cite[Theorem 73.4, p. 440]{munkres3}: if $K \subseteq \mathbb{R}^3$ is a compact set and $\{N_k\}$ is a decreasing $p$--neighbourhood basis of $K$, then $\check{H}^d(K)$ is (isomorphic to) the direct limit of the direct sequence \[H^d(N_1) \stackrel{\varphi_1}{\longrightarrow} H^d(N_2) \stackrel{\varphi_2}{\longrightarrow} H^d(N_3) \stackrel{\varphi_3}{\longrightarrow} \ldots,\] where $\varphi_k$ is the homomorphism induced in cohomology by the inclusion $N_{k+1} \subseteq N_k$.
\smallskip

We already know, by Example \ref{ejem:cell}, that if $K$ is cellular then $r(K) = 0$. The interesting content Theorem \ref{teo:null} is the converse implication. If $K$ is a continuum with $r(K) = 0$, then Proposition \ref{prop:easy1} implies that $K$ has a neighbourhood basis $\{N_k\}$ comprised of polyhedral manifolds with ${\rm rk}\ H_1(N_k) = 0$, or equivalently $H_1(N_k) = 0$ for every $k$ (recall that we are taking coefficients in $\mathbb{Z}_2$). Since $K$ is connected, we can discard those components of the $N_k$ that do not meet $K$, obtaining smaller neighbourhoods which we again call $N_k$. These are still a neighbourhood basis of $K$ with $H_1(N_k) = 0$ for every $k$, but they are now connected. Summing up, we have proved the following result:

\begin{proposition} \label{prop:null_trivial} If $K \subseteq \mathbb{R}^3$ is a continuum with $r(K) = 0$, then $K$ has a neighbourhood basis $\{N_k\}$ comprised of connected polyhedral manifolds with $H_1(N_k) = 0$ for every $k$.
\end{proposition}

The proof of Theorem \ref{teo:null} consists in showing that the $N_k$ mentioned in Proposition \ref{prop:null_trivial} are ``balls with holes'' (an idea that we formalise below under the name of \emph{perforated balls}) and then using the condition that $\check{H}^2(K) = 0$ to prove that these balls with holes can be ``filled in'' to obtain a new neighbourhood basis $\{\hat{N}_k\}$ of $K$ comprised of actual balls, thus showing that $K$ is cellular. 

\begin{definition} \label{def:bola_perforada} Suppose $B \subseteq \mathbb{R}^3$ is a polyhedral $3$--ball whose interior contains a finite number of disjoint polyhedral $3$--balls $B_1, \ldots, B_m$. Then we call $N := B - \bigcup_{i=1}^m {\rm int}\ B_i$ a \emph{perforated ball}. We denote $\hat{N} := B$ and say that $\hat{N}$ is obtained from $B$ by \emph{filling its holes}.
\end{definition}

The proofs in this section make use of certain duality relations in homology and cohomology combined with the universal coefficient theorem. Namely, we will use the following:
\begin{enumerate}
	\item Lefschetz duality: if $N$ is a compact $3$--manifold, then \[H_d(N,\partial N) = H_{3-d}(N).\]
	\item Alexander duality: if $L$ is a compact subset of $\mathbb{S}^3$, then \[\tilde{H}_d(\mathbb{S}^3 - L) = \check{H}^{2-d}(L).\] When $L$ is a polyhedron or a manifold its \v{C}ech cohomology agrees with its singular cohomology, so using the universal coefficient theorem it follows that \[\tilde{H}_d(\mathbb{S}^3-L) = H_{2-d}(L).\]
	\item Alexander duality in manifolds with boundary: if $N$ is a compact $3$--manifold, possibly with boundary, and $L \subseteq {\rm int}\ N$ is compact, then \[\check{H}^d(N,L) = H_{3-d}(N-L,\partial N).\] As in (2), if $L$ is a polyhedron or a manifold then \[H_d(N,L) = H_{3-d}(N-L,\partial N).\]
\end{enumerate}
The first two are standard and well known, but maybe this is not the case of the third. Thus we have included a proof in an Appendix. Alexander duality (2) holds more generally for compact subsets of any $\mathbb{S}^n$, with $2-d$ replaced by $(n-1)-d$.
\smallskip

We will also use the \emph{polyhedral Sch\"onflies theorem} in $\mathbb{R}^3$, due to Alexander \cite[Theorem 12, p. 122]{moise2}. A \emph{sphere} is a set homeomorphic to the standard sphere $\mathbb{S}^2 := \{p \in \mathbb{R}^{3} : \|p\| = 1\}$.

\medskip
{\it Theorem.} (Alexander) Let $S \subseteq \mathbb{R}^3 \subseteq \mathbb{S}^3$ be a polyhedral sphere. Then $\mathbb{S}^3 - S$ has exactly two components $U$ and $V$, both of which have the following properties: ({\it i}\/) their topological frontiers ${\rm fr}\ U$ and ${\rm fr}\ V$ coincide with $S$, ({\it ii}\/) their closures $\overline{U}$ and $\overline{V}$ are homeomorphic to a ball.
\medskip

This theorem has a long history and a great significance in the development of topology. We shall say some words about it in Section \ref{sec:esferas} but, for the moment, we will just carry on with the proof of the nullity property.

\begin{lemma} \label{lem:perfora} Let $N$ be a compact, connected, polyhedral $3$--manifold in $\mathbb{R}^3$.
\begin{enumerate}
	\item[({\it i}\/)] If $H_1(N) = 0$ and $H_2(N) = 0$ then $N$ is ball,
	\item[({\it ii}\/)] If $H_1(N) = 0$ then $N$ is a perforated ball.
\end{enumerate}
\end{lemma}
\begin{proof} It is best to think of $N$ as a subset of $\mathbb{S}^3$, and this is what we will do for the proof. Also, we consider $\mathbb{S}^3$ as $\mathbb{R}^3$ together with a point at infinity, denoted $\infty$.
\smallskip

({\it i}\/) From the long exact sequence in reduced homology for the pair $(N,\partial N)$ and Lefschetz duality it follows that \[\tilde{H}_d(\partial N) = H_{d+1}(N,\partial N) = H_{2-d}(N) = \left\{ \begin{array}{cc} 0 & \text{ if } d = 0 \\ 0 & \text{ if } d = 1 \\ \mathbb{Z}_2 & \text{ if } d = 2 \end{array} \right.\] and therefore $\partial N$ is a connected compact surface (without boundary) with the homology of the $2$--sphere. Hence $\partial N$ is a $2$--sphere.

By the polyhedral Sch\"onflies theorem in $\mathbb{S}^3$, $\mathbb{S}^3 - \partial N$ has exactly two connected components $U$ and $V$, the closure of each of them being homeomorphic to a $3$--ball. Now observe that $\mathbb{S}^3 - \partial N$ is the disjoint union of the open sets ${\rm int}\ N$ and $\mathbb{S}^3 - N$. Since ${\rm int}\ N$ is connected because $N$ was assumed to be connected, it follows that ${\rm int}\ N$ is one of the components of $\mathbb{S}^3 - \partial N$. Thus ${\rm int}\ N = U$ (say) and consequently $N = \overline{{\rm int}\ N} = \overline{U}$ is homeomorphic to a $3$--ball.
\smallskip

({\it ii}\/) By Alexander duality $\tilde{H}_d(\mathbb{S}^3 - N) = H_{2-d}(N) = 0$ for $d = 1,2$, which shows that each component of the compact polyhedral $3$--manifold $\mathbb{S}^3 - {\rm int}\ N$ has trivial reduced homology. This implies, by part ({\it i}\/) of this lemma, that they are all polyhedral balls. Since $N$ is a compact subset of $\mathbb{R}^3$, it follows that $\infty$ must belong to the interior of one of those balls, say $B_{\infty}$, and the remaining balls $B_1, \ldots, B_m$ are subsets of $\mathbb{R}^3$. Therefore \[N = (\mathbb{S}^3 - {\rm int}\ B_{\infty}) - \bigcup_{i=1}^m\ {\rm int}\ B_i,\] and denoting $B := \mathbb{S}^3 - {\rm int}\ B_{\infty}$, which is a polyhedral ball in $\mathbb{R}^3$, we have \[N = B - \bigcup_{i=1}^m\ {\rm int}\ B_i,\] where the $B_i$ are disjoint balls in the interior of $B$. Hence $N$ is a perforated ball.
\end{proof}

It follows from Proposition \ref{prop:null_trivial} and Lemma \ref{lem:perfora} that if $K \subseteq \mathbb{R}^3$ is a continuum with $r(K) = 0$, then it has a neighbourhood basis $\{N_k\}$ comprised of perforated balls. Under the additional hypothesis that $\check{H}^2(K) = 0$ we will show that the holes of the $N_k$ can be filled in and, although this enlarges the $N_k$, we \emph{still} obtain a neighbourhood basis $\{\hat{N}_k\}$ of $K$. Now $\{\hat{N}_k\}$ is comprised of balls, thus proving that $K$ is cellular. Lemma \ref{lem:bound} provides the geometric basis for performing this operation.

We need to recall a definition. A closed connected surface $S$ has $H_2(S) = \mathbb{Z}_2$ by Poincar\'e duality, and a generator for $H_2(S)$ is called a \emph{fundamental class} of $S$. We denote it by $[S]$. If one is willing to think in terms of simplicial homology and imagines $S$ as a triangulated surface, $[S]$ is nothing but the sum of all the $2$--simplices of $S$.

\begin{lemma} \label{lem:bound} Let $N$ be a compact polyhedral $3$--manifold in $\mathbb{R}^3$. Assume that $N$ is connected and $S \subseteq {\rm int}\ N$ is a polyhedral $2$--sphere such that $[S] = 0$ in $H_2(N)$. Then the ball bounded by $S$ in $\mathbb{R}^3$ is contained in $N$.
\end{lemma}
\begin{proof} From the exact sequence \[H_3(N) = 0 \longrightarrow H_3(N,S) \longrightarrow H_2(S) = \mathbb{Z}_2 \stackrel{0}{\longrightarrow} H_2(N)\] it follows that $H_3(N,S) = \mathbb{Z}_2$. By Alexander duality in manifolds with boundary this implies that \[\mathbb{Z}_2 = H_3(N,S) = H_0(N - S,\partial N)\] which shows that $N - S$ has exactly one connected component disjoint from $\partial N$. Denote this component $U$. On the one hand, $U$ is closed in $N - S$ so it is also closed in $\mathbb{R}^3 - S$. On the other hand, since $U$ is open in $N - S$ and does not meet $\partial N$, it is open in ${\rm int}\ N - S$ and consequently also in $\mathbb{R}^3 - S$. Therefore $U$ must be a component of $\mathbb{R}^3 - S$ and, being contained in $N$ (which is compact), it has to be the bounded one; hence by the Sch\"onflies theorem it is the ball bounded by $S$ in $\mathbb{R}^3$.
\end{proof}

Let \[G_1 \stackrel{\varphi_1}{\longrightarrow} G_2 \stackrel{\varphi_2}{\longrightarrow} \ldots \stackrel{\varphi_{k-1}}{\longrightarrow} G_k \stackrel{\varphi_k}{\longrightarrow} \ldots\] be a direct sequence of abelian groups and let $G$ be its direct limit (see any algebra book or \cite[pp. 434ff.]{munkres3} for a definition of these concepts). We denote $\varphi_{k\ell} := \varphi_{\ell-1} \circ \ldots \circ \varphi_{k+1} \circ \varphi_k$ for $k \leq \ell$ and $\psi_k : G_k \longrightarrow G$ the canonical maps from each of the $G_k$ to the direct limit $G$. It is an immediate consequence of the definition of $G$ that for an element $g_k \in G_k$, the image $\psi_k(g_k) = 0 \in G$ if and only if there exists $\ell \geq k$ such that $\varphi_{k \ell}(g_k) = 0 \in G_{\ell}$. This generalises to the following lemma:

\begin{lemma} \label{lem:zero} Let \[G_1 \stackrel{\varphi_1}{\longrightarrow} G_2 \stackrel{\varphi_2}{\longrightarrow} \ldots \stackrel{\varphi_{k-1}}{\longrightarrow} G_k \stackrel{\varphi_k}{\longrightarrow} \ldots\] be a direct sequence whose direct limit $G$ is the trivial group. If $G_k$ is finitely generated, then there exists $\ell \geq k$ such that $\varphi_{k\ell} = 0$.
\end{lemma}

We omit its very simple proof and proceed with the proof of Theorem \ref{teo:null}.

\begin{proof}[Proof of Theorem \ref{teo:null}] By Proposition \ref{prop:null_trivial} and Lemma \ref{lem:perfora}.{\it ii} we see that $K$ has a neighbourhood basis $\{N_k\}$ comprised of perforated balls. After passing to a subsequence if needed we may assume that $N_{k+1} \subseteq {\rm int}\ N_k$ for every $k$. 

The \v{C}ech cohomology group $\check{H}^2(K)$ is the direct limit of the sequence \[H^2(N_1) \longrightarrow H^2(N_2) \longrightarrow H^2(N_3) \longrightarrow \ldots,\] where the arrows are the inclusion induced homomorphisms. The assumption that $\check{H}^2(K) = 0$ together with the fact that $H^2(N_k)$ is finitely generated for each $k$ imply by Lemma \ref{lem:zero} that for every $k$ there exists $\ell \geq k$ such that the inclusion $N_{\ell} \subseteq N_k$ induces the zero homomorphism $0 : H^2(N_k) \longrightarrow H^2(N_{\ell})$. Thus by passing to a subsequence of the $N_k$ we may assume that \[H^2(N_1) \stackrel{0}{\longrightarrow} H^2(N_2) \stackrel{0}{\longrightarrow} H^2(N_3) \stackrel{0}{\longrightarrow} \ldots\] and so by duality each inclusion $N_{k+1} \subseteq N_k$ induces the zero homomorphism $0 : H_2(N_{k+1}) \longrightarrow H_2(N_k)$.

Consider $\hat{N}_{k+1}$, the ball obtained from the perforated ball $N_{k+1}$ by capping its holes. The sphere $\partial \hat{N}_{k+1}$ is clearly one of the boundary components of $N_{k+1}$ and therefore its fundamental class $[\partial \hat{N}_{k+1}]$ is taken to zero in $H_2(N_k)$ by the inclusion $N_{k+1} \subseteq N_k$. Thus by Lemma \ref{lem:bound} it follows that $\hat{N}_{k+1} \subseteq N_k$. Consequently $\{\hat{N}_k\}$ is a neighbourhood basis of $K$ comprised of balls, and we conclude that $K$ is cellular.
\end{proof}

\section{The subadditivity property} \label{sec:sub}

Suppose $K \subseteq \mathbb{R}^3$ is a compact set expressed as the union of two compact sets $K_1$ and $K_2$. We call such an expression a \emph{decomposition} of $K$. The subadditivity property reads as follows:

\begin{theorem} \label{teo:sub} (Subadditivity) Assume that $K = K_1 \cup K_2$ is a tame decomposition of $K$ and suppose that $\check{H}^1(K_1 \cap K_2) = 0$. Then \[r(K_1) + r(K_2) \leq r(K).\]
\end{theorem}

The tameness condition on the decomposition $K = K_1 \cup K_2$ is described in the next subsection. Roughly, it guarantees that the decomposition $K = K_1 \cup K_2$ can be realized geometrically, which is necessary to perform the \emph{splitting construction} that is the basis of the subadditivity property. Given a $pm$--neighbourhood $N$ of $K$, the splitting construction produces two $pm$--neighbourhoods $N_1$ and $N_2$ of $K_1$ and $K_2$ respectively such that ${\rm rk}\ H_1(N_1) +  {\rm rk}\ H_1(N_2) \leq {\rm rk}\ H_1(N)$. We describe this construction informally in the second subsection and then more carefully in the third subsection.
\smallskip

The inequality in Theorem \ref{teo:sub} may be strict. For instance, denote $F_1$ and $F_2$ each of the two symmetric halves of the Fox--Artin arc $F$ shown in Figure \ref{fig:wild_2}. Both $F_1$ and $F_2$ are cellular \cite[Example 1.2, p. 983]{foxartin1}, so $r(F_1) = r(F_2) = 0$. However $F$ is not cellular, and we have the strict inequality $r(F) \geq 1 > 0 + 0 = r(F_1) + r(F_2)$.

\subsection{Tame decompositions} We need to introduce some notation:

\begin{enumerate}
	\item[] $Q_1$ is the parallelepiped $[-1,1] \times [-1,1] \times [-1,0]$,
	\item[] $Q_2$ is the parallelepiped $[-1,1] \times [-1,1] \times [0,1]$,
	\item[] $Q$ is the cube $Q_1 \cup Q_2$,
	\item[] $S$ is the square $Q_1 \cap Q_2 = [-1,1] \times [-1,1] \times \{0\}$,
	\item[] $\dot{S}$ denotes the square $S$ minus its edges.
\end{enumerate}

A schematic view of all these items is shown in Figure \ref{fig:tameaxes}. The $z$ axis has been chosen to be horizontal and contained in the plane of the paper for notational convenience.

\begin{figure}[h]
\begin{pspicture}(0,0)(7.4,5)
%\psgrid(0,0)(7.4,5)
	\rput[bl](0,0){\scalebox{0.6}{\includegraphics{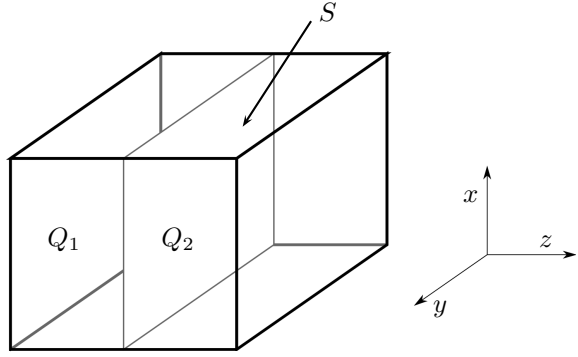}}}
	\psline{->}(4,4.4)(3.1,3) \rput[lb](4.1,4.4){$S$}
	\rput(0.75,1.5){$Q_1$} \rput(2.25,1.5){$Q_2$}
	\rput[bl](7,1.4){$z$} \rput[b](6.1,2){$x$} \rput[t](5.7,0.7){$y$}
\end{pspicture}
\caption{The setup for Definition \ref{def:tame} \label{fig:tameaxes}}
\end{figure}

\begin{definition} \label{def:tame} A decomposition $K = K_1 \cup K_2$ is \emph{tame} if
\begin{enumerate}
	\item[({\it i}\/)] $K \cap S = K_1 \cap K_2 \subseteq \dot{S}$,
	\item[({\it ii}\/)] $K_1 \cap Q \subseteq Q_1$ and $K_2 \cap Q \subseteq Q_2$.
\end{enumerate}
For the sake of convenience we shall widen our definition slightly and say that a decomposition $K = K_1 \cup K_2$ is tame if there exists an ambient homeomorphism $h : \mathbb{R}^3 \longrightarrow \mathbb{R}^3$ that takes $K$, $K_1$ and $K_2$ onto sets that satisfy ({\it i}\/) and ({\it ii}\/) above.
\end{definition}

The intuitive content of Definition \ref{def:tame} is, loosely speaking, that $S$ realises \emph{geometrically} the purely set-theoretical decomposition $K = K_1 \cup K_2$. The part of $K$ that lies in $Q$ is structured in two ``halves'', $K \cap Q_1$ and $K \cap Q_2$. The first one sits in $Q_1$ and, because of condition ({\it ii}\/), is comprised \emph{exclusively} of points from $K_1$. The second one sits in $Q_2$ and is comprised \emph{exclusively} of points from $K_2$.

\begin{example} \label{ejem:tame} The prototypical example of a tame decomposition is as follows. Let $H \subseteq \mathbb{R}^3$ be the hyperplane \[H := \{(x,y,z) \in \mathbb{R}^3 : z = 0\}\] and denote $H_1$ and $H_2$ the two closed halfspaces into which $H$ separates $\mathbb{R}^3$. For definitiness say \[H_1 := \{(x,y,z) \in \mathbb{R}^3 : z \leq 0\}\ \ \text{ and }\ \ H_2 := \{(x,y,z) \in \mathbb{R}^3 : z \geq 0\}.\]

Let $K \subseteq \mathbb{R}^3$ be a compact set that meets $H$ and let $K_1 := K \cap H_1$ and $K_2 := K \cap H_2$. Then $K = K_1 \cup K_2$ is a tame decomposition. Indeed, after scaling $K$ down with an ambient homeomorphism we may assume that $K_1 \cap K_2 \subseteq \dot{S}$. Then conditions ({\it i}\/) and ({\it ii}\/) in Definition \ref{def:tame} are trivially fulfilled.
\end{example}

Example \ref{ejem:tame} provides a good picture to have in mind and is enough for simple si{\-}tuations such as the decompositions considered in the proof of Theorem \ref{teo:wild_arc}. However, a useful definition of tameness should be of a local nature. For instance, consider the decomposition $B^* = B^*_1 \cup B^*_2$ shown in Figure \ref{fig:dec_B}. $B^*_1$ and $B^*_2$ meet neatly along a disk, and locally (near that disk) the situation is just as the one described in Example \ref{ejem:tame}, so we want to say that this decomposition is tame. But far away from the disk the sets $B^*_1$ and $B^*_2$ become entangled, so it is not possible (not even after performing an ambient homeomorphism) to have $B^*_1$ contained in $H_1$ and $B^*_2$ contained in $H_2$. This is why a local version of Example \ref{ejem:tame} is needed, and that is exactly what motivates Definition \ref{def:tame}. Notice that according to this definition $B^* = B^*_1 \cup B^*_2$ is indeed a tame decomposition.
\smallskip

A decomposition may not be tame either because $K_1 \cap K_2$ cannot be included in a square or because any cut along $K_1 \cap K_2$ does not separate $K$ locally into $K_1$ and $K_2$ lying on different sides of $K_1 \cap K_2$. The first possibility will be illustrated later on in Example \ref{exam:not_tame2}. The second one is easier and we give an example now.

\begin{example} \label{ejem:not_tame1} Consider the continuum $K$ shown in Figure \ref{fig:not_tame1}. It is the union of two continua $K_1$ and $K_2$, shown respectively in dark and light gray, whose intersection is the black square at the back of the figure, barely visible.

\begin{figure}[h]
\begin{pspicture}(0,0)(6.8,5.7)
%\psgrid(0,0)(6.8,5.7)
	\rput[bl](0,0){\scalebox{0.75}{\includegraphics{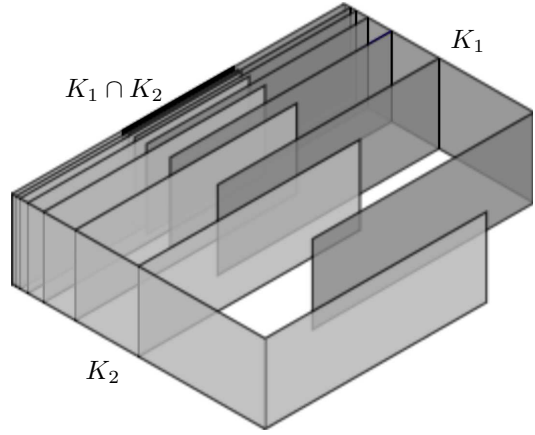}}}
	\rput(1.2,0.8){$K_2$} \rput(6,5.2){$K_1$} \rput[br](2,4.4){$K_1 \cap K_2$}
\end{pspicture}
\caption{A decomposition that is not tame \label{fig:not_tame1}}
\end{figure}

It is clearly possible to find a homeomorphism that takes $K_1 \cap K_2$ into the square $S$ in such a way that the first condition of Definition \ref{def:tame} is fulfilled. We leave it to the reader to show that it is not possible to achieve the second condition of Definition \ref{def:tame}.
\end{example}

Notice that the subadditivity property holds in Example \ref{ejem:not_tame1} (clearly $r(K_1) = r(K_2) = r(K) = 0$) even though the decomposition $K = K_1 \cup K_2$ is not tame. Later on, in Example \ref{exam:not_tame2}, we will present a decomposition that is not tame and where the subadditivity property does not hold.

\subsection{The splitting construction (informal version)} This is the geometric basis for the subadditivity property. It may be summarized as follows:

\begin{theorem} \label{teo:aux} (Notation and hypotheses as in Theorem \ref{teo:sub}) Given $V_1$ and $V_2$ open neighbourhoods of $K_1$ and $K_2$, there exists an open neighbourhood $V$ of $K$ such that any $pm$--neighbourhood $N$ of $K$ contained in $V$ can be used to construct two $pm$--neighbourhoods $N_1$ and $N_2$ of $K_1$ and $K_2$ with the following properties:
\begin{enumerate}
	\item[($S_1)$] ${\rm rk}\ H_1(N_1) + {\rm rk}\ H_1(N_2) \leq {\rm rk}\ H_1(N)$,
	\item[($S_2)$] $N_1 \subseteq V_1$ and $N_2 \subseteq V_2$.
\end{enumerate}
\end{theorem}
\begin{proof}[Proof of Theorem \ref{teo:sub} from Theorem \ref{teo:aux}.] If $r(K) = \infty$ there is nothing to prove, for the inequality $r(K_1) + r(K_2) \leq r(K)$ holds automatically. So assume $r(K) < \infty$ and choose a $pm$--neighbourhood basis $\{N_k$\} of $K$ such that ${\rm rk}\ H_1(N_k) = r(K)$ for every $k$. This exists by Proposition \ref{prop:easy1}.{\it ii}.

Using Theorem \ref{teo:aux} each $N_k$ gives rise to two $pm$--neighbourhoods $N_{k,1}$ and $N_{k,2}$ of $K_1$ and $K_2$ such that \[{\rm rk}\ H_1(N_{k,1}) + {\rm rk}\ H_1(N_{k,2}) \leq r(K)\] for every $k$. Moreover, by property ($S_2$) we may assume that $\{N_{k,1}\}$ and $\{N_{k,2}\}$ are neighbourhood bases of $K_1$ and $K_2$ respectively. Hence by Proposition \ref{prop:easy1}.{\it i}. \[r(K_1) \leq \liminf_{k \rightarrow \infty}\ {\rm rk}\ H_1(N_{k,1}) \ \ \text{ and } \ \ r(K_2) \leq \liminf_{k \rightarrow \infty}\ {\rm rk}\ H_1(N_{k,2}).\]

Thus \begin{multline*} r(K_1) + r(K_2) \leq \liminf_{k \rightarrow \infty}\ {\rm rk}\ H_1(N_{k,1}) + \liminf_{k \rightarrow \infty}\ {\rm rk}\ H_1(N_{k,2}) \\ \leq \liminf_{k \rightarrow \infty}\ ({\rm rk}\ H_1(N_{k,1}) + {\rm rk}\ H_1(N_{k,2})) \leq r(K)\end{multline*} which proves Theorem \ref{teo:sub}.
\end{proof}

The splitting process behind Theorem \ref{teo:aux} has some delicate points. In order to make the essential ideas easier to grasp we first describe it informally here and postpone the details to the next section.

For simplicity let us set ourselves in the situation of Example \ref{ejem:tame}. Thus the compact set $K$ is split in two halves $K_1 = K \cap H_1$ and $K_2 = K \cap H_2$ by the vertical hyperplane $H : z = 0$. Let $N$ be a small $pm$--neighbourhood of $K$. We are going to show how to construct $N_1$ and $N_2$ from $N$.
\smallskip

\fbox{Step 1} We split $N$ along $H$ letting $M_1 := N \cap H_1$ and $M_2 := N \cap H_2$. Clearly $K_1 \subseteq M_1$ and $K_2 \subseteq M_2$, but notice $M_1$ and $M_2$ are not neighbourhoods of $K_1$ and $K_2$ because the points of $K_1 \cap K_2$ do not belong to the interior of $M_1$ or $M_2$. For technical reasons it is convenient that $N \cap H$ be a $2$--manifold, which can always be achieved (this will be proved later on). See Figure \ref{fig:corte1}.{\it a}.
\smallskip

{\bf Notation.} It is a well known consequence of the Jordan--Sch\"onflies theorem that a compact connected $2$--manifold $C$ contained in the plane is a \emph{disk with holes}. More precisely, there exist a disk $\hat{C}$ and mutually disjoint disks $D_1, \ldots, D_r \subseteq {\rm int}\ \hat{C}$ such that $C = \hat{C} - \bigcup_{i=1}^r {\rm int}\ D_i$. Observe that $\hat{C}$ is the smallest disk that contains $C$. We call the set $C^* := \bigcup_{i=1}^r D_i$ the \emph{capping set} for $C$; by its very definition it is a disjoint union of disks with the property that $C \cup C^* = \hat{C}$. Notice that $C \cap C^*$ is the union of all the boundary components of $C$ except for the outermost one. The process of recovering the disk $\hat{C}$ from $C$ by taking its union with $C^*$ is usually called \emph{capping the holes} of $C$.
\smallskip

\fbox{Step 2} Consider the intersection $N \cap H$. Because it is a compact $2$--manifold contained in a plane (namely $H$), each of its connected components $C_1, C_2, \ldots, C_r$ is a disk with holes. For illustrative purposes let us assume that $N \cap H$ has exactly four connected components: two disks $C_1$ and $C_2$, an annulus $C_3$, and a disk with two holes $C_4$ (see Figure \ref{fig:corte1}.b). The labelling of the components is not entirely arbitrary and requires some care (details will be given later on).

\begin{figure}[h]
\begin{pspicture}(0,0)(10.4,6.2)
%\psgrid(0,0)(10.4,6.2)
	\rput[bl](0,0){\scalebox{0.6}{\includegraphics{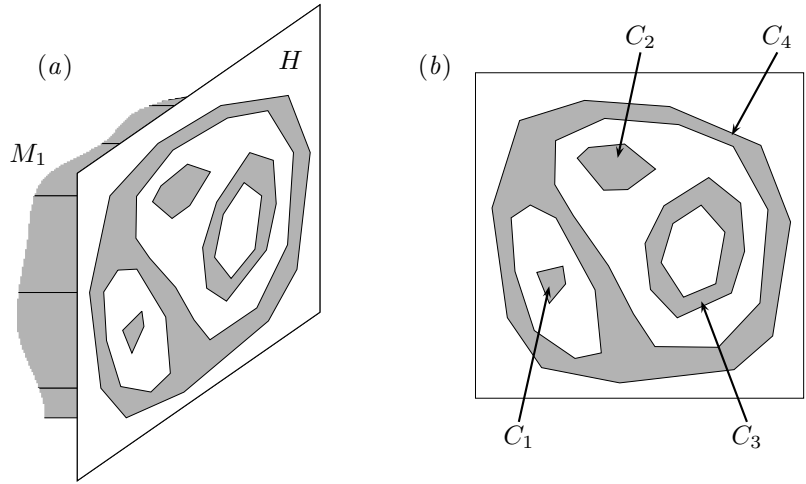}}}
	\psline{->}(6.6,0.8)(7,2.6) \rput(6.6,0.6){$C_1$}
	\psline{->}(8.2,5.7)(7.9,4.3) \rput(8.2,5.9){$C_2$}
	\psline{->}(9.6,0.8)(9,2.4) \rput(9.6,0.6){$C_3$}
	\psline{->}(10,5.7)(9.4,4.6) \rput(10,5.9){$C_4$}
	\rput[br](0.4,4.2){$M_1$} \rput(3.6,5.6){$H$}
	\rput(0.5,5.5){$({\it a}\/)$} \rput(5.5,5.5){$({\it b}\/)$}
\end{pspicture}
\caption{Cutting $N$ along $H$ \label{fig:corte1}}
\end{figure}

Notice how the components $C_i$ may well be nested. This is precisely what makes the construction somewhat delicate. 
\smallskip

{\bf Notation.} Suppose $a<b$ are a real numbers and $A$ is a subset of $H$. We shall allow ourselves some freedom and denote \[A \times [a,b] = \{(x,y,z) : (x,y,0) \in A \text{ and } a \leq z \leq b\}.\] If the reader thinks of $H$ as the $xy$ plane and $\mathbb{R}^3$ as $H \times \mathbb{R}$, the above notation is self explanatory. Geometrically, $A$ is a $2$--dimensional object and $A \times [a,b]$ is a $3$--dimensional object that results from \emph{extruding} $A$ along the $z$ axis.
\smallskip

\fbox{Step 3} We are going to construct $N_1$ by suitably enlarging $M_1$ in a two stage process. The construction of $N_2$ from $M_2$ is entirely analogous and it will not be described explicitly.

{\it First stage.} The intersection $M_1 \cap H$ is the same as $N \cap H$, so it consists of the four components $C_1, C_2, C_3$ and $C_4$ described in the previous step. At this first stage we simply extrude each $C_i$ along the $z$ axis an amount $\nicefrac{i}{4}$. Namely, we enlarge $M_1$ to $\hat{M}_1$ defined as \[\hat{M}_1 := M_1 \cup (C_1 \times [0,\nicefrac{1}{4}]) \cup (C_2 \times [0,\nicefrac{2}{4}]) \cup (C_3 \times [0,\nicefrac{3}{4}]) \cup (C_4 \times [0,1]).\] Figure \ref{fig:extru}.{\it a} shows the extrusion of just the first three components. The final result can be seen in Figure \ref{fig:extru}.{\it b}.

\begin{figure}[h]
\begin{pspicture}(0,0)(11,6.4)
%\psgrid(0,0)(11,6.4)
	\rput[bl](0,0){\scalebox{0.6}{\includegraphics{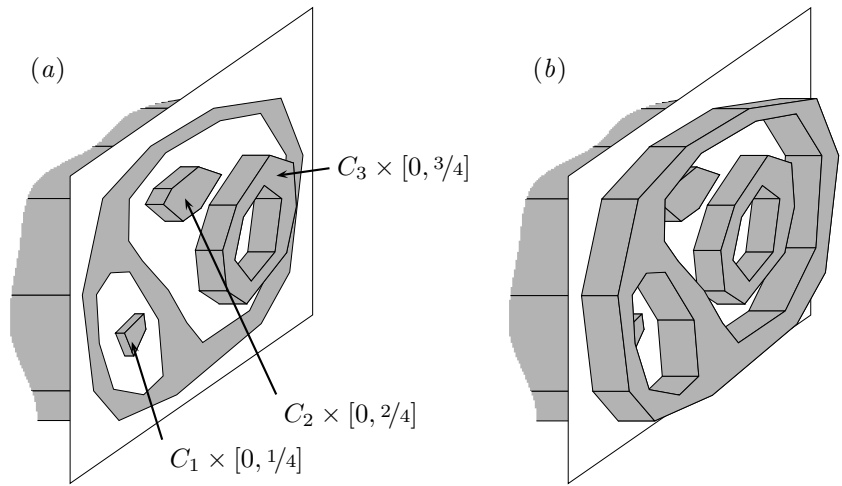}}}
	\psline{->}(2,0.6)(1.6,1.9) \rput[tl](2.1,0.5){$C_1 \times [0,\nicefrac{1}{4}]$}
	\psline{->}(3.5,1.2)(2.3,3.8) \rput[tl](3.6,1.1){$C_2 \times [0,\nicefrac{2}{4}]$}
	\psline{->}(4.2,4.2)(3.5,4.1) \rput[l](4.3,4.2){$C_3 \times [0,\nicefrac{3}{4}]$}
	\rput(0.5,5.5){$({\it a}\/)$} \rput(7.1,5.5){$({\it b}\/)$}
\end{pspicture}
\caption{The extrusion process \label{fig:extru}}
\end{figure}

{\it Second stage.} Formally we should start with $C_1$ and $C_2$ but, because they have no holes, the construction we are about to describe is trivial for them. Hence we consider the annulus $C_3$ and its capping set $C_3^*$, which is a single disk. We attach a thickened copy of $C_3^*$ at the right end of $C_3 \times [0,\nicefrac{3}{4}]$, as a plug at the end of a pipe. Figure \ref{fig:cap1} shows the set \[\hat{M}_1 \cup (C_3^* \times [\nicefrac{3}{4}-\nicefrac{1}{8},\nicefrac{3}{4}]),\] which in panel ({\it a}\/) has been pictured out of its position for clarity and in panel ({\it b}\/) has been drawn at its definitive location, neatly fitting at the end of $C_3 \times [0,\nicefrac{3}{4}]$.

\begin{figure}[h]
\begin{pspicture}(0,0)(11,11)
%\psgrid(0,0)(11,11)
	\rput[bl](0,0){\scalebox{0.6}{\includegraphics{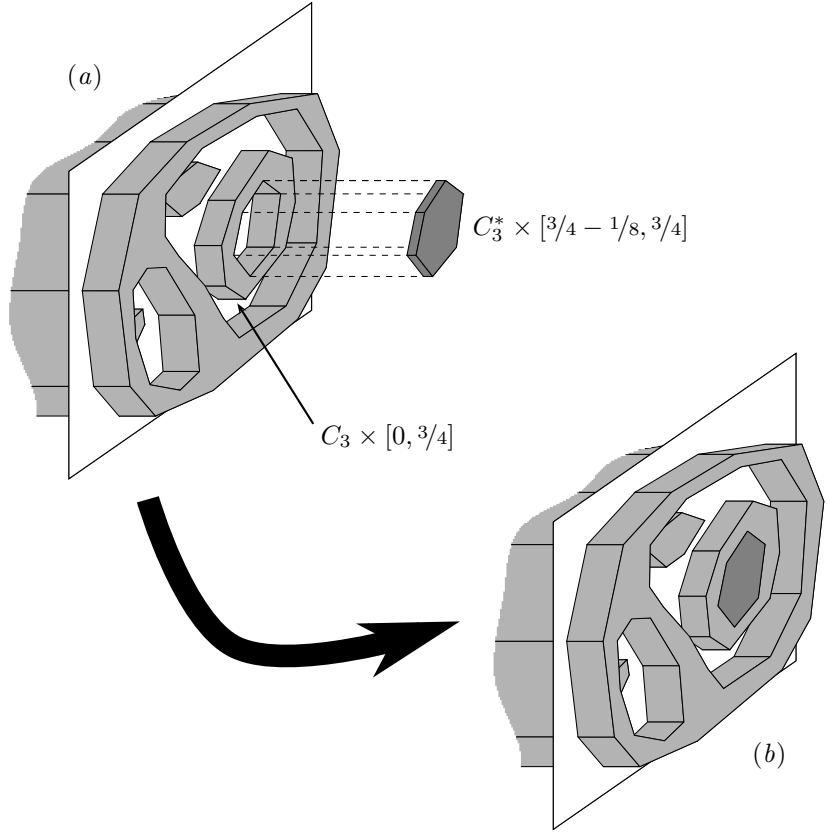}}}
	\psline{->}(4,5.4)(3,7) \rput[tl](4.1,5.4){$C_3 \times [0,\nicefrac{3}{4}]$}
	\rput[l](6.1,8){$C^*_3 \times [\nicefrac{3}{4}-\nicefrac{1}{8},\nicefrac{3}{4}]$}
	\rput(1,10){({\it a}\/)} \rput(10,1){({\it b}\/)}
\end{pspicture}
	\caption{Capping the end of the extruded $C_3$ \label{fig:cap1}}
\end{figure}

Notice that $C_3^* \times [\nicefrac{3}{4}-\nicefrac{1}{8},\nicefrac{3}{4}]$ is indeed a thickened copy of $C_3^*$ aligned flush with the right end of $C_3 \times [0,\nicefrac{3}{4}]$ because both extend up to the plane $z=\nicefrac{3}{4}$. For the sake of brevity we will sometimes refer to these sets as the \emph{thickened} $C_3^*$ and the \emph{extruded} $C_3$, respectively.

The same process has to be done now with $C_4$, whose capping set $C_4^*$ consists of two disks. Figure \ref{fig:cap2} shows the set \[N_1 := \hat{M}_1 \cup (C_3^* \times [\nicefrac{3}{4}-\nicefrac{1}{8},\nicefrac{3}{4}]) \cup (C_4^* \times [1-\nicefrac{1}{8},1])\] which is precisely the manifold $N_1$ that we were looking for. This finishes the splitting construction.

\begin{figure}[h]
\begin{pspicture}(0,0)(11,11)
%\psgrid(0,0)(11,11)
	\rput[bl](0,0){\scalebox{0.6}{\includegraphics{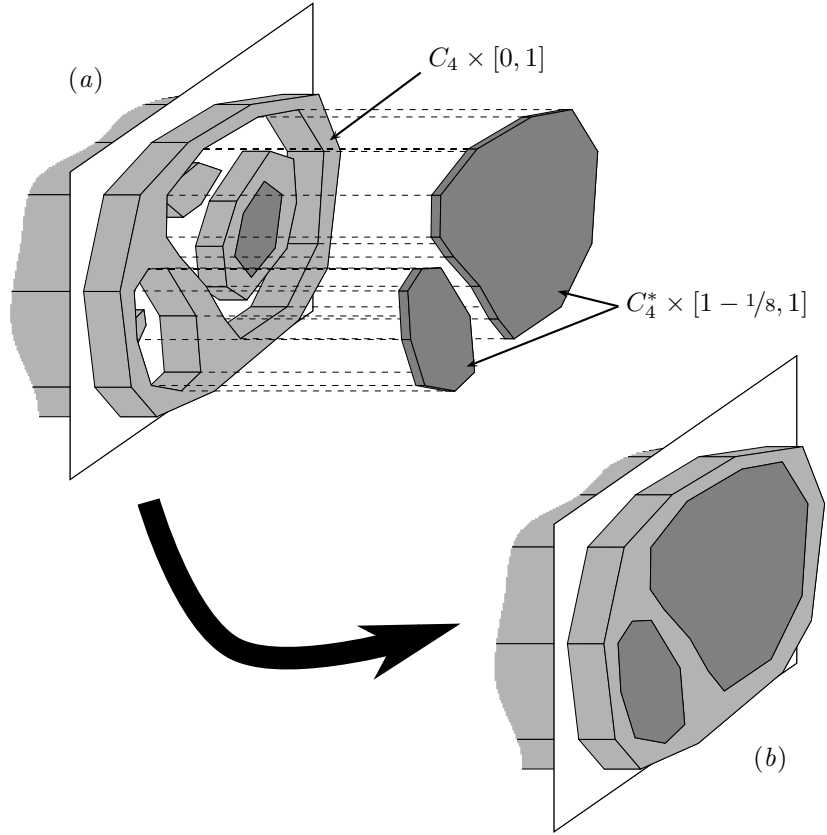}}}
	\psline{->}(5.4,10.1)(4.2,9.2) \rput[bl](5.5,10.1){$C_4 \times [0,1]$}
	\psline{->}(8,7)(6,6.2)\psline{->}(8,7)(7.2,7.2) \rput[lt](8.1,7.2){$C_4^* \times [1-\nicefrac{1}{8},1]$}
	%\rput[l](6.1,8){$C^*_3 \times [\nicefrac{3}{4}-\nicefrac{1}{8},\nicefrac{3}{4}]$}
	\rput(1,10){({\it a}\/)} \rput(10,1){({\it b}\/)}
\end{pspicture}
	\caption{Capping the holes in the extruded $C_4$ \label{fig:cap2}}
\end{figure}

Observe that the amount of extrusion of the $C_i$ is smaller for the innermost components and bigger for the outermost ones. This is important to guarantee that, as is the case in our example, the thickened $C_4^*$ intersects neither the extruded $C_1$ and $C_2$ nor the extruded and capped $C_3$ (it is instructive to think what would have happened if $C_3$ had been extruded farther to the right than $C_4$). This geometric fact is required to prove the inequality \[{\rm rk}\ H_1(N_1) + {\rm rk}\ H_1(N_2) \leq {\rm rk}\ H_1(N),\] which is where the subadditivity property ultimately stems from.

\subsection{The splitting construction (formal version)} We now abandon the simple situation just considered and set ourselves in the case of a completely general tame decomposition $K = K_1 \cup K_2$.
\smallskip

\fbox{Step 1} Recall that a subset $X$ of $Y$ is collared in $Y$ if there exists an embedding $c : X \times [0,1) \longrightarrow Y$ onto an open neighbourhood of $X$ in $Y$ and such that $c(x,0) = x$ for every $x \in X$ (see for instance the paper by Brown \cite[p. 332]{brown4}).

\begin{lemma} \label{lem:transverse} Let $N \subseteq \mathbb{R}^3$ be a compact polyhedral $3$--manifold such that $N \cap S \subseteq \dot{S}$. Assume that $N$ has no vertices in $S$. Then $N \cap S$ is a compact $2$--manifold that is collared both in $N \cap Q_1$ and in $N \cap Q_2$.
\end{lemma}
\begin{proof} $\partial N$ is a union of triangles that are pairwise disjoint or meet in a single common edge or a single common vertex. Since $\partial N$ has no vertices in $S$, each of its triangles meets $S$ in a straight segment; thus $\partial N \cap S$ is a union of straight segments any two of which are either disjoint or meet in a single common endpoint. Each edge in $\partial N$ belongs to exactly two triangles because $\partial N$ has no boundary, so each vertex in $\partial N \cap S$ belongs to exactly two segments. Consequently $\partial N \cap S$ is a disjoint union of polygonal simple closed curves. Since $\partial N \cap S$ is precisely the topological frontier of $N \cap S$ in $S$, we conclude that the latter is a compact $2$--manifold with boundary.

We now prove that $N \cap S$ is collared in $N \cap Q_1$ (the argument for $N \cap Q_2$ being entirely analogous). By a theorem of Brown \cite[Theorem 1, p. 337]{brown4}, which was given a simpler proof by Connelly \cite{connelly1}, it is enough to show that $N \cap S$ is \emph{locally collared} in $N \cap Q_1$. That is, we need to show that every $p \in N \cap S$ has a neighbourhood $U$ in $N \cap S$ that is collared in $N \cap Q_1$. But this is a straightforward consequence of the fact that each simplex in the triangulation of $N$ meets $S$ transversally.
\end{proof}

\begin{proposition} \label{prop:prep0} Suppose $K = K_1 \cup K_2$ is a tame decomposition. Then $K$ has arbitrarily small $pm$--neighbourhoods $N$ such that:
\begin{enumerate}
	\item[({\it i}\/)] $N \cap S$ is a compact $2$--manifold,
	\item[({\it ii}\/)] $N \cap S$ is collared in $N \cap Q_1$ and in $N \cap Q_2$.
\end{enumerate}
Furthermore, if $r(K) < \infty$ then one can achieve
\begin{enumerate}
	\item[({\it iii}\/)] ${\rm rk}\ H_1(N) = r(K)$.
\end{enumerate}
\end{proposition}
\begin{proof} Let $U$ be a neighbourhood of $K$ and pick a $pm$--neighbourhood $N$ of $K$ contained in $U$. Since $K \cap S \subseteq \dot{S}$ because of tameness, we can take $N$ so small that $N \cap S \subseteq \dot{S}$. If $r(K) < \infty$, by Proposition \ref{prop:equivalentes} we are entitled to assume that ${\rm rk}\ H_1(N) = r(K)$.

To obtain ({\it i}\/) and ({\it ii}\/) we only need to show that we can choose $N$ satisfying the hypothesis of Lemma \ref{lem:transverse}; namely, that $N$ does not have any vertices in $S$. Since $N$ has only finitely many vertices, there are arbitrarily small values of $\delta \geq 0$ such that $\delta$ is different from the $z$--coordinates of all the vertices of $N$. If $t_{\delta} : \mathbb{R}^3 \longrightarrow \mathbb{R}^3$ denotes the translation $t_{\delta}(x,y,z) := (x,y,z-\delta)$, then $t_{\delta}(N)$ is a polyhedral manifold that has no vertices in $S$ so by Lemma \ref{lem:transverse} conditions ({\it i}\/) and ({\it ii}\/) are satisfied. A judicious choice of sufficiently small $\delta$ will also guarantee that $t_{\delta}(N)$ is still a neighbourhood of $K$ contained in $U$. Clearly $t_{\delta}(N)$ is homeomorphic to $N$ and so ${\rm rk}\ H_1(t_{\delta}(N)) = {\rm rk}\ H_1(N) = r(K)$, which shows that ({\it iii}\/) is also satisfied.
\end{proof}

\begin{proposition} \label{prop:prep1} Suppose $K = K_1 \cup K_2$ is a tame decomposition. Let $U_1$ and $U_2$ be neighbourhoods of $K_1$ and $K_2$ respectively. Then $K$ has arbitrarily small $pm$--neighbourhoods $N$ that can be written as $N = M_1 \cup M_2$, where:
\begin{enumerate}
	\item[({\it i}\/)] $K_1 \subseteq M_1 \subseteq U_1$ and $K_2 \subseteq M_2 \subseteq U_2$,
	\item[({\it ii}\/)] $M_1 \cap M_2 = N \cap S \subseteq \dot{S}$,
	\item[({\it iii}\/)] $M_1 \cap Q \subseteq Q_1$ and $M_2 \cap Q \subseteq Q_2$,
	\item[({\it iv}\/)] $M_1$ and $M_2$ are compact polyhedral $3$--manifolds,
	\item[({\it v}\/)] $N \cap S$ is a $2$--manifold that is collared in $M_1$ and $M_2$.
\end{enumerate}
Furthermore, if $r(K) < \infty$ then one can achieve
\begin{enumerate}
	\item[({\it vi}\/)] ${\rm rk}\ H_1(N) = r(K)$.
\end{enumerate}
\end{proposition}
\begin{proof} Possibly after reducing $U_1$ and $U_2$ we may assume that they are polyhedral and so small that $U_1 \cap U_2 \subseteq {\rm int}\ Q$, $U_1 \cap \partial Q \subseteq Q_1$ and $U_2 \cap \partial Q \subseteq Q_2$.

Let $V := (U_1 - {\rm int}\ Q_2) \cup (U_2 - {\rm int}\ Q_1)$. We claim that $V$ is a neighbourhood of $K$. Pick a point $p \in K$, and assume first that $p \in K_1 - K_2$. Then $p \not\in Q_2$ and so it has an open neighbourhood $W$ contained in $U_1$ and disjoint from $Q_2$; hence $W \subseteq V$. The same holds true for $p \in K_2 - K_1$. Finally, let $p \in K_1 \cap K_2$. Then $p \in S$ and it has an open neighbourhood $W$ contained in $U_1 \cap U_2$. Since \[W = (W - {\rm int}\ Q_2) \cup (W - {\rm int}\ Q_1)\] it follows that \[W \subseteq (U_1 - {\rm int}\ Q_2) \cup (U_2 - {\rm int}\ Q_1) = V.\]

Let $N$ be a $pm$--neighbourhood of $K$ contained in $V$ and so small that $N \cap S \subseteq U_1 \cap U_2$. We can assume that $N$ satisfies conditions ({\it i}\/), ({\it ii}\/) and (if $r(K) < \infty$) ({\it iii}\/) of Proposition \ref{prop:prep0}. Let $M_1 := (N \cap U_1) - {\rm int}\ Q_2$ and $M_2 := (N \cap U_2) - {\rm int}\ Q_1$.
\smallskip

({\it i}\/) Since $K_1 \cap {\rm int}\ Q_2 = \emptyset$, we see that $K_1 \subseteq M_1$. Also $M_1 \subseteq U_1$ by construction. Similarly one proves that $K_2 \subseteq M_2 \subseteq U_2$.
\smallskip

({\it ii}\/) Clearly $M_1 \cap M_2 = (N \cap U_1 \cap U_2) - ({\rm int}\ Q_1 \cup {\rm int}\ Q_2)$. Notice that, since $U_1 \cap U_2 \subseteq {\rm int}\ Q$, it follows that $M_1 \cap M_2 \subseteq N \cap S$. To prove the reverse inclusion, recall that we assumed $N \cap S \subseteq U_1 \cap U_2$. Since $N \cap S$ is disjoint from ${\rm int}\ Q_1$ and ${\rm int}\ Q_2$, it follows that $N \cap S \subseteq M_1 \cap M_2$.
\smallskip

({\it iii}\/) Observe that $M_1 \cap Q = (N \cap U_1) \cap (Q_1 \cup \partial Q)$. The assumption that $U_1 \cap \partial Q \subseteq Q_1$ then implies that $M_1 \cap Q = (N \cap U_1) \cap Q_1 \subseteq Q_1$. Similarly one proves $M_2 \cap Q \subseteq Q_2$.
\smallskip

({\it iv}\/) $M_1$ and $M_2$ are polyhedra because they are the intersection of the polyhedron $N$ with the polyhedra $U_1 - {\rm int}\ Q_2$ and $U_2 - {\rm int}\ Q_1$. It remains to show that they are also $3$--manifolds. Observe that $M_1 \cap S = N \cap S$. Since $N \cap S$ is collared in $N \cap Q_1$, there is a neighbourhood $V_1$ of $N \cap S$ in $N \cap Q_1$ homeomorphic to $(N \cap S) \times (-1,0]$. Also $N \cap S$ is a $2$--manifold, and so it follows that $V_1$ is $3$--manifold (with boundary). $M_1 - S$ is a $3$--manifold because it coincides with $N - M_2$, which is open in the $3$--manifold $N$. As $M_1$ is covered by the interiors of $V_1$ and $M_1 - S$, we conclude that $M_1$ is a $3$--manifold. The same argument works for $M_2$.
\end{proof}

\fbox{Step 2} Suppose that $C_1, \ldots, C_r$ are the connected components of a compact $2$--manifold contained in $\mathbb{R}^2$. For $i \neq j$ we have $\partial \hat{C}_i \cap \partial \hat{C}_j \subseteq C_i \cap C_j = \emptyset$, so the disks $\hat{C}_i$ and $\hat{C}_j$ have to satisfy precisely one of the following three alternatives: either $\hat{C}_i \subseteq {\rm int}\ \hat{C}_j$, or $\hat{C}_j \subseteq {\rm int}\ \hat{C}_i$, or $\hat{C}_i \cap \hat{C}_j = \emptyset$. In the first case we shall say that $C_i$ is \emph{interior} to $C_j$, in the second one that $C_j$ is \emph{interior} to $C_i$ and in the third one that $C_i$ and $C_j$ are \emph{indifferent}. The following result is very easy to prove by induction:

\begin{lemma} \label{lem:orden} The connected components of a compact $2$--manifold contained in $\mathbb{R}^2$ may be labeled $C_1, \ldots, C_r$ in such a way that if $C_i$ is interior to $C_j$ then $i < j$.
\end{lemma}

\fbox{Step 3} Now we are ready to define $N_1$ and $N_2$. Denote by $C_1, \ldots, C_r$ the components of $N \cap S$ and by $C^*_i$ their capping sets, as usual. By Lemma \ref{lem:orden} we can assume that the $C_i$ are ordered in such a way that if $C_i$ is interior to $C_j$ then $i < j$. Fix $0 < \varepsilon < 1$, which later on we will need to adjust.

{\it First stage.} Let \[\hat{M}_1 := M_1 \cup \bigcup_{i=1}^r \left(C_i \times \left[0,\varepsilon \frac{i}{r}\right]\right).\] As explained earlier, $\hat{M}_1$ results from the extrusion of the $C_i$ along the $z$ axis. The innermost components of $N \cap S$ (those with smaller $i$) are extended only a little, whereas the outermost components are extended farther.

\begin{remark} \label{rem:prep4} There exist homeomorphisms $h_1 : M_1 \longrightarrow \hat{M}_1$ and $h_2 : M_2 \longrightarrow \hat{M}_2$ such that $h_1(p,0) = \left(p,\frac{i}{r}\right)$ for $p \in C_i$ and similarly $h_2(p,0) = \left(p,-\frac{i}{r}\right)$ for $p \in C_i$.
\end{remark}

Remark \ref{rem:prep4} is a fairly standard exercise in using the collar of $N \cap S$ in $M_1$ and $M_2$, which we omit. The interested reader can find the detailed argument for a similar situation in a paper by Sikkema \cite[Theorem 2, p. 400]{sikkema1}.
\smallskip

{\it Second stage.} Let \[P_1 := \bigcup_{i=1}^r \left(C^*_i \times \left[\varepsilon\left(\frac{i}{r} - \frac{1}{2r}\right),\varepsilon \frac{i}{r}\right]\right)\] and, finally, define \[N_1 := \hat{M}_1 \cup P_1.\]

As before, we call $C_i \times \left[0,\varepsilon \frac{i}{r}\right]$ the \emph{extruded} $C_i$ and $C^*_i \times \left[\varepsilon\left(\frac{i}{r} - \frac{1}{2r}\right),\varepsilon \frac{i}{r}\right]$ the \emph{thickened} $C^*_i$.

\begin{remark} \label{rem:prep2} The thickened $C^*_i$ meet $\hat{M}_1$ in a disjoint union of annuli, precisely one for each of the thickened disks in $C^*_i$.
\end{remark}

Remark \ref{rem:prep2} should be geometrically clear, but nevertheless we prove it formally. Obviously the thickened $C^*_i$ are disjoint from $M_1$, so it is enough to study their intersection with the extruded $C^*_j$. Assume that \[\emptyset \neq \left( C^*_i \times \left[\varepsilon\left( \frac{i}{r} - \frac{1}{2r}\right),\varepsilon\frac{i}{r} \right]\right) \bigcap \left(C_j \times \left[0,\varepsilon \frac{j}{r}\right]\right)\] or, equivalently, \[\emptyset \neq (C^*_i \cap C_j) \times \left( \left[\varepsilon\left( \frac{i}{r}-\frac{1}{2r}\right),\varepsilon \frac{i}{r} \right] \bigcap \left[0,\varepsilon \frac{j}{r} \right]\right).\]

Each of the factors of the product has to be nonempty; in particular $C^*_i \cap C_j \neq \emptyset$. Suppose $i \neq j$. Then $\partial C^*_i \cap \partial C_j \subseteq \partial C_i \cap \partial C_j = \emptyset$ and, since $C_j$ is connected, it must be wholly contained in one the disks whose union is $C^*_i$. As $\hat{C}_j$ is the smallest disk that contains $C_j$, we conclude that $\hat{C}_j \subseteq C^*_i \subseteq {\rm int}\ \hat{C}_i$, so that $C_j$ is interior to $C_i$ and thus $j < i$ because of the choice of the ordering $C_1, \ldots, C_r$. Then $j \leq i-1$ which implies that the intervals $\left[\varepsilon\left( \frac{i}{r}-\frac{1}{2r}\right),\varepsilon\frac{i}{r} \right]$ and $\left[0,\varepsilon\frac{j}{r} \right]$ are disjoint. Hence the above intersection is empty, a contradiction, and we conclude that if the thickened $C^*_i$ meets the extruded $C_j$, then $i = j$. In that case their intersection is \[\left(C^*_i \times \left[\varepsilon\left( \frac{i}{r} - \frac{1}{2r}\right),\varepsilon\frac{i}{r} \right]\right) \bigcap \left(C_i \times \left[0,\varepsilon \frac{i}{r}\right]\right) = \partial C^*_i \times \left[\varepsilon\left( \frac{i}{r}-\frac{1}{2r}\right),\varepsilon \frac{i}{r}\right],\] which is just a disjoint union of annuli, one for each component of $C^*_i$. The claim of Remark \ref{rem:prep2} follows.
\smallskip

An analogous process has to be performed on $M_2$. Thus we let \[\hat{M}_2 := M_2 \cup \bigcup_{i=1}^r \left(C_i \times \left[ -\varepsilon \frac{i}{r},0 \right]\right),\] \[P_2 := \bigcup_{i=1}^r \left(C_i^* \times \left[ -\varepsilon \frac{i}{r},-\varepsilon \left(\frac{i}{r} + \frac{1}{2r}\right) \right]\right),\] and \[N_2 := \hat{M}_2 \cup P_2.\]

\subsection{The proof of Theorem \ref{teo:aux}} To finish this section we put all the pieces together. First we show that $N_1$ and $N_2$ satisfy property ($S_1$) of Theorem \ref{teo:aux}.

\begin{proposition} \label{P1} $N_1$ and $N_2$ have property ($S_1$) \[{\rm rk}\ H_1(N_1) + {\rm rk}\ H_1(N_2) \leq {\rm rk}\ H_1(N).\]
\end{proposition}
\begin{proof} Consider the compact $3$--manifold $N'$ obtained from the disjoint union $N_1 \coprod N_2$ by identifying each $\left(p,\frac{i}{r}\right) \in \hat{C}_i \times \left\{\frac{i}{r}\right\} \subseteq N_1$ with its corresponding $\left(p,-\frac{i}{r}\right) \in \hat{C}_i \times \left\{-\frac{i}{r}\right\} \subseteq N_2$ via an equivalence relation $\sim$. This process cannot generally be performed in $\mathbb{R}^3$, but of course it can be done in an abstract way.

Identifying $N_1$, $N_2$ and all of their subsets with their images in $N'$ we may write $N' = N_1 \cup N_2$, where $N_1 \cap N_2$ is a disjoint union of disks. In particular $H_1(N_1 \cap N_2) = 0$ and then from the Mayer--Vietoris exact sequence \[H_1(N_1 \cap N_2) \longrightarrow H_1(N_1) \oplus H_1(N_2) \longrightarrow H_1(N')\] we see that the arrow connecting the last two terms in the sequence is injective, so ${\rm rk}\ H_1(N_1) + {\rm rk}\ H_1(N_2) \leq {\rm rk}\ H_1(N')$.

Recall that $N_1 = \hat{M}_1 \cup P_1$ and $N_2 = \hat{M}_2 \cup P_2$. Let us denote by $\hat{M}$ and $P$ the result of performing on $\hat{M}_1 \coprod \hat{M}_2$ and $P_1 \coprod P_2$, respectively, the identifications prescribed by $\sim$. Clearly $N' = \hat{M} \cup P$.

Each component of $P$ is a $3$--ball, so $H_1(P) = 0$. It follows from Remark \ref{rem:prep2} that each component of $\hat{M} \cap P$ is an annulus that is contained in one of the components of $P$, and no component of $P$ contains more than one of these annuli. Hence the inclusion induced map $H_0(\hat{M} \cap P) \longrightarrow H_0(P)$ is injective. Thus in the Mayer--Vietoris exact sequence \[H_1(\hat{M}) \oplus H_1(P) = H_1(\hat{M}) \longrightarrow H_1(N') \longrightarrow H_0(\hat{M} \cap P) \longrightarrow H_0(\hat{M}) \oplus H_0(P)\] we see that the rightmost arrow is injective, so the leftmost one is surjective and consequently ${\rm rk}\ H_1(N') \leq {\rm rk}\ H_1(\hat{M})$.

The two homeomorphisms $h_1$ and $h_2$ mentioned in Remark \ref{rem:prep4} can be pasted together to produce a homeomorphism $h : N \longrightarrow \hat{M}$. In particular, ${\rm rk}\ H_1(N) = {\rm rk}\ H_1(\hat{M})$. Together with the two other inequalities obtained earlier, we have that \[{\rm rk}\ H_1(N_1) + {\rm rk}\ H_1(N_2) \leq {\rm rk}\ H_1(N') \leq {\rm rk}\ H_1(\hat{M}) = {\rm rk}\ H_1(N),\] as was to be proved.
\end{proof}

There is a hypothesis for the subadditivity theorem that we have not used yet, namely that $\check{H}^1(K_1 \cap K_2) = 0$. It is only now, to show that $N_1$ and $N_2$ satisfy the smallness condition ($S_2$) of Theorem \ref{teo:aux}, that this hypothesis comes into play. We want to restate it in a more convenient fashion that follows immediately from Alexander duality in $\mathbb{S}^2$.

\begin{remark} $K_1 \cap K_2$ does not separate the square $S$.
\end{remark}

\begin{lemma} \label{lem:disks} Let $L$ be a nonseparating compact subset of $\mathbb{R}^2$. Then $L$ has arbitrarily small neighbourhoods that are finite unions of disjoint compact disks.
\end{lemma}
\begin{proof} By the frame theorem \cite[Theorem 6, p. 72]{moise2} $L$ has arbitrarily small $m$--neighbourhoods $E$ with the property that different components of $\mathbb{R}^2 - E$ lie in different components of $\mathbb{R}^2 - L$. Since $\mathbb{R}^2 - L$ is connected by hypothesis, $\mathbb{R}^2 - E$ is connected too. The components of $E$ are disks with holes but, since $E$ does not disconnect $\mathbb{R}^2$, it follows that they are actually disks.
\end{proof}

\begin{proposition} \label{P2} Property ($S_2$) holds: given $V_1$ and $V_2$ open neighbourhoods of $K_1$ and $K_2$, there exists $V$ open neighbourhood of $K$ such that if $N \subseteq V$ one can achieve, by choosing $\varepsilon > 0$ sufficiently small at Step 3, that $N_1 \subseteq V_1$ and $N_2 \subseteq V_2$.
\end{proposition}
\begin{proof} This will follow rather easily once we establish the following
\smallskip

{\it Claim.} If $W$ is a neighbourhood of $N \cap S$, an appropriate choice of $\varepsilon > 0$ in Step 3 yields \[N_1 \subseteq M_1 \cup W \ \ \text{ and } \ \ N_2 \subseteq M_2 \cup W.\]

{\it Proof.} By Lemma \ref{lem:disks} there is a finite union of disjoint compact disks $E$ that is a neighbourhood of $K \cap S$ and is contained in $N \cap S$, and consequently also in $W$. Choose $0 < \varepsilon < 1$ such that $E \times [-\varepsilon,\varepsilon] \subseteq W$.

Pick any $C_i$. Since $\hat{C}_i$ contains both $C_i$ and $C^*_i$ (actually, it is the union of both sets), we have \[\bigcup_{i=1}^r \left( C_i \times \left[0,\varepsilon \frac{i}{r}\right] \right) \subseteq \hat{C}_i \times [0,\varepsilon]\] and \[P_1 = \bigcup_{i=1}^r \left(C^*_i \times \left[\varepsilon \left(\frac{i}{r} - \frac{1}{2r}\right),\varepsilon \frac{i}{r}\right]\right) \subseteq \hat{C}_i \times [0,\varepsilon].\] Consequently $N_1 \subseteq M_1 \cup (\hat{C}_i \times [0,\varepsilon])$. 
Now, $C_i$ is a connected subset of $E$, which is a union of disks; hence $C_i$ is contained in a disk $D \subseteq E$. Since $\hat{C}_i$ is the smallest disk that contains $C_i$, it follows that $\hat{C}_i \subseteq D \subseteq E$. Hence \[N_1 \subseteq M_1 \cup (\hat{C}_i \times [0,\varepsilon]) \subseteq M_1 \cup (E \times [0,\varepsilon]) \subseteq M_1 \cup W.\]

An analogous argument works for $N_2$.
\smallskip

We can now finish the proof of the proposition. Clearly there exists an open neighbourhood $V$ of $K$ such that if $N \subseteq V$ then $M_1 \subseteq V_1$ and $M_2 \subseteq V_2$. Letting $W := V_1 \cap V_2$, which is a neighbourhood of $M_1 \cap M_2 = N \cap S$, and applying the claim above there exists $\varepsilon > 0$ such that $N_1 \subseteq M_1 \cup W$ and $N_2 \subseteq M_2 \cup W$. Thus \[N_1 \subseteq M_1 \cup W \subseteq V_1 \cup W = V_1 \ \ \text{and} \ \ N_2 \subseteq M_2 \cup W \subseteq V_2 \cup W = V_2.\]
\end{proof}

\section{Balls and spheres that cannot be attractors} \label{sec:esferas}

The classical \emph{Jordan curve theorem} states that a simple closed curve $\gamma$ in the sphere $\mathbb{S}^2$ separates it in exactly two connected components $U_1$ and $U_2$, called the \emph{complementary domains} of $\gamma$, whose topological boundaries coincide with $\gamma$. This turns out to be much more general: a connected, compact $n$--manifold $M \subseteq \mathbb{S}^{n+1}$ separates $\mathbb{S}^{n+1}$ in exactly two complementary domains whose topological boun{\-}da{\-}ries coincide with $M$. The proof is homological in nature and depends on Alexander duality (see the argument before Proposition \ref{prop:surface} and the proof of Lemma \ref{lem:borde}).
\smallskip

The \emph{Sch\"onflies theorem} improves on the Jordan curve theorem by being more precise about the nature of $\overline{U}_1$ and $\overline{U}_2$. Namely, in two dimensions it states that the closures of the two complementary domains of a simple closed curve $\gamma \subseteq \mathbb{S}^2$ are disks. Alexander tried to generalise this result to higher dimensions and proved it for polyhedral spheres $S$ in $\mathbb{S}^3$ (we have used this result in proving Lemma \ref{lem:perfora} earlier). However, he also discovered that there exist non polyhedral spheres for which the Sch\"onflies theorem is not true. In a beautiful paper \cite{alexander1} he described an embedding $S^*$ of the sphere in $\mathbb{S}^3$ (the \emph{horned sphere of Alexander}) such that one of its complementary domains is not simply connected, so its closure cannot be homeomorphic to a ball. This shows that the Sch\"onflies theorem is false, in general, in $\mathbb{S}^3$. The closure $B^*$ of the other complementary domain \emph{is} homeomorphic to a ball, which we call the \emph{Alexander ball} (see Figure \ref{fig:bola_alexander}). In this section we shall prove that neither $S^*$ nor $B^*$ can be attractors.

\subsection{The ball of Alexander cannot be an attractor} We briefly recall how $B^*$ is defined. Start with a graph $\Gamma$ as shown in Figure \ref{fig:bola_alexander_construye}.{\it a}. It is a dyadic tree that keeps branching towards its set of limit points, which is a Cantor set $C$. Consider now this very same tree, but embedded differently in $\mathbb{R}^3$, like $\Gamma^*$ in Figure \ref{fig:bola_alexander_construye}.{\it b}. At each branching stage every pair of innermost branches get tangled. 

\begin{figure}[h!]
\begin{pspicture}(0,-0.7)(12.2,4)
%\psgrid(0,-0.7)(12.2,4)
	\rput[bl](0,0){\includegraphics{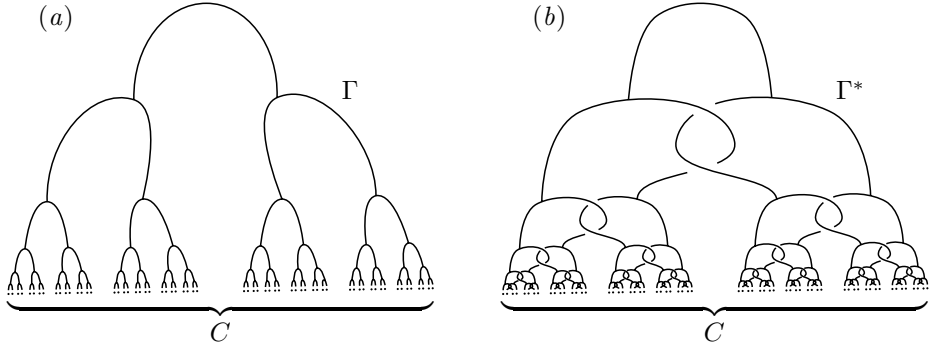}}
	\rput[bl](0.4,3.5){({\it a}\/)} \rput[bl](6.9,3.5){({\it b}\/)}
	\rput[bl](4.4,2.6){$\Gamma$} \rput[bl](10.9,2.6){$\Gamma^*$}
	\rput[bl](0,-0.3){$\underbrace{\hspace{5.6cm}}$} \rput[b](2.8,-0.6){$C$}
	\rput[bl](6.5,-0.3){$\underbrace{\hspace{5.6cm}}$} \rput[b](9.3,-0.6){$C$}
\end{pspicture}
\caption{The construction of Alexander's ball $B^* \subseteq \mathbb{R}^3$ \label{fig:bola_alexander_construye} (first stages)}
\end{figure}

Finally, let $B^*$ be a sort of tubular neighbourhood of $\Gamma^*$ whose diameter keeps decreasing as we approach the limit points of $\Gamma^*$, that is, the Cantor set $C$. The diameter of $B^*$ at those points is zero; $B^*$ has pointy tips at $C$. A schematic picture of $B^*$ is shown in Figure \ref{fig:bola_alexander}. One can prove that $B^*$ is homeomorphic to a ball so that in particular $\partial B^*$ is homeomorphic to a sphere, but $\mathbb{R}^3 - B^*$ is not simply connected because it is homeomorphic to $\mathbb{R}^3 - \Gamma^*$, which is not simply connected \cite[\S 1, pp. 619 and 620]{blankinshipfox1}. In particular, $B^*$ is not cellular.

\begin{figure}
\begin{center}
	\includegraphics{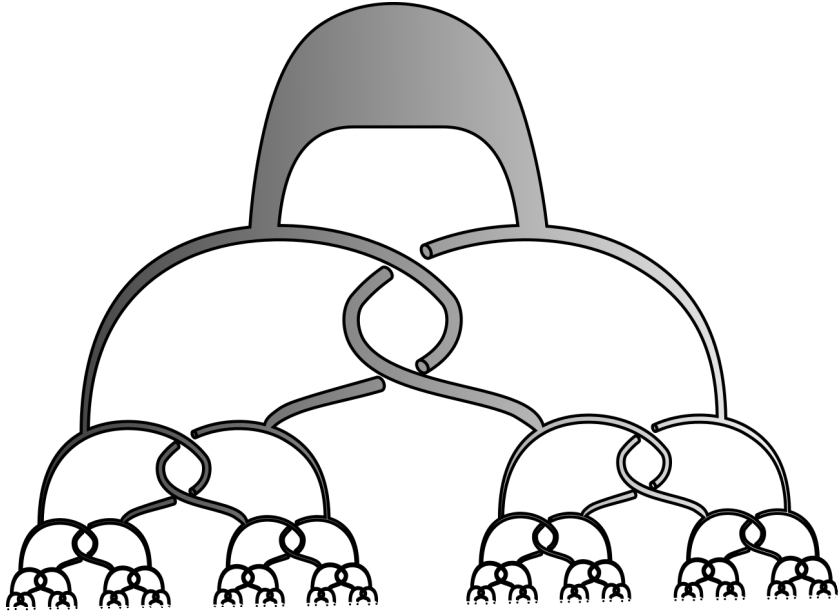}
\end{center}
\caption{The ball of Alexander $B^*$ \label{fig:bola_alexander} (final stage)}
\end{figure}

\begin{theorem} \label{teo:ball} The ball of Alexander $B^*$ cannot be an attractor.
\end{theorem}
\begin{proof} It is enough to show that $r(B^*) = \infty$. Begin by writing $B^*$ as the union $B^* = B^*_1 \cup B^*_2$ as shown in Figure \ref{fig:dec_B} below. This decomposition is clearly tame. Figure \ref{fig:halves_homeo} suggests how to prove that both $B^*_1$ and $B^*_2$ are ambient homeomorphic to $B^*$.

\begin{figure}
\begin{center}
\begin{pspicture}(0,0)(11,8.2)
	\rput[bl](0,0){\includegraphics{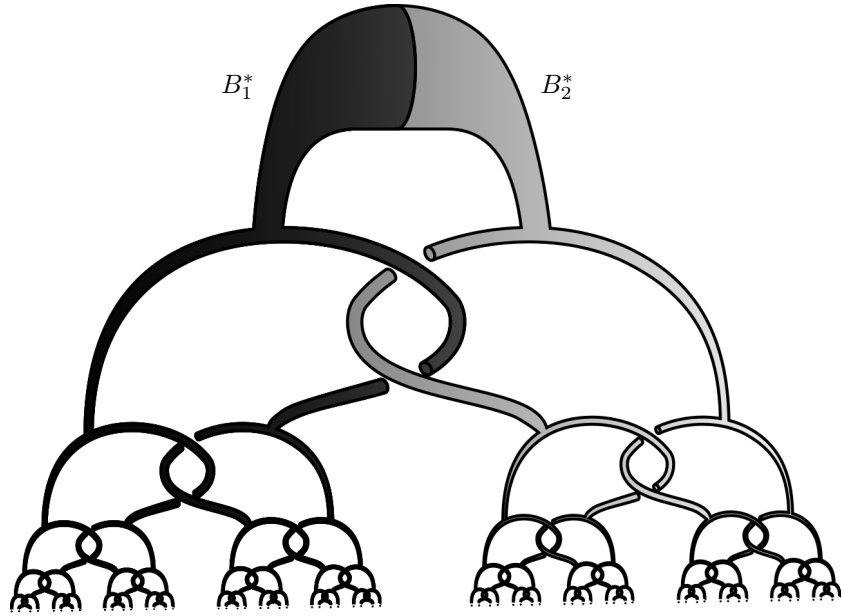}}
	\rput(3,7){$B^*_1$} \rput(7.2,7){$B^*_2$}
\end{pspicture}
\end{center}
\caption{A tame decomposition of $B^*$ \label{fig:dec_B}}
\end{figure}

\begin{figure}
\begin{center}
	\includegraphics{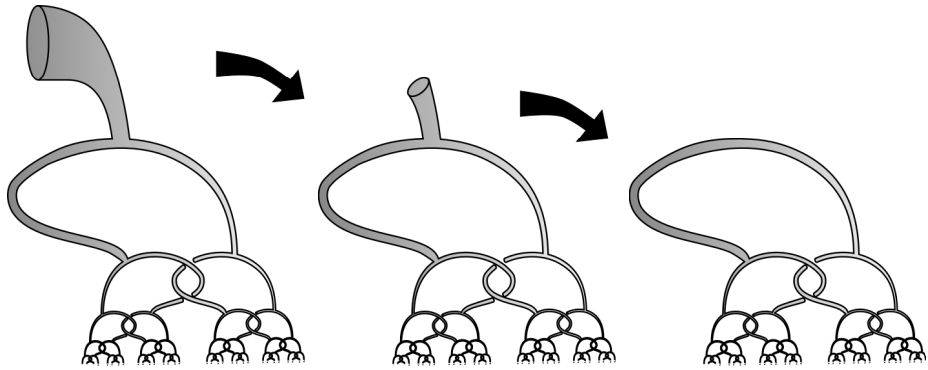}
\end{center}
\caption{Showing that $B^*_2$ is ambient homeomorphic to $B^*$ \label{fig:halves_homeo}}
\end{figure}

We have $r(B^*) = r(B^*_1) = r(B^*_2)$ by invariance, and also $r(B^*_1) + r(B^*_2) \leq r(B^*)$ by subadditivity. Therefore \[2r(B^*) \leq r(B^*),\] which implies that either $r(B^*) = 0$ or $r(B^*) = \infty$. Since $B^*$ is not cellular, by nullity we have $r(B^*) \neq 0$. Hence $r(B^*) = \infty$ and we are finished.
\end{proof}

Theorem \ref{teo:ball} illustrates again how a perfectly acceptable set, such as a ball, can be embedded in $\mathbb{R}^3$ in such a way that it cannot be an attractor.
\medskip

We stated earlier, while dealing with the subadditivity property, that it does not generally hold for decompositions that are not tame. Now we are in a position to prove this by example.

\begin{example} \label{exam:not_tame2} There is a decomposition of the standard unit ball $K$ into two continua $K_1$ and $K_2$ that meet in a disk and such that $r(K_1) + r(K_2) \not\leq r(K)$. Thus the subadditivity property does not hold for this decomposition.
\end{example}
\begin{proof} Start with the unit ball $K$. To define $K_1$, refer to Figure \ref{fig:not_tame} below and dig a hole in $K$ following the pattern of the ball of Alexander. $K_2$ is the part of $K$ that is dug out in this process, that is, $K_2 := \overline{K - K_1}$. Although $K_1$ and $K_2$ are shown separately in Figure \ref{fig:not_tame} for clarity purposes, we remark that $K_2$ actually fills in the hole in $K_1$. Notice that $K_2$ is ambient homeomorphic to the ball of Alexander.

\begin{figure}[h]
\begin{pspicture}(0,0)(12.6,4)
	%\psgrid(0,0)(12.6,4)
	\rput[bl](0,0){\includegraphics{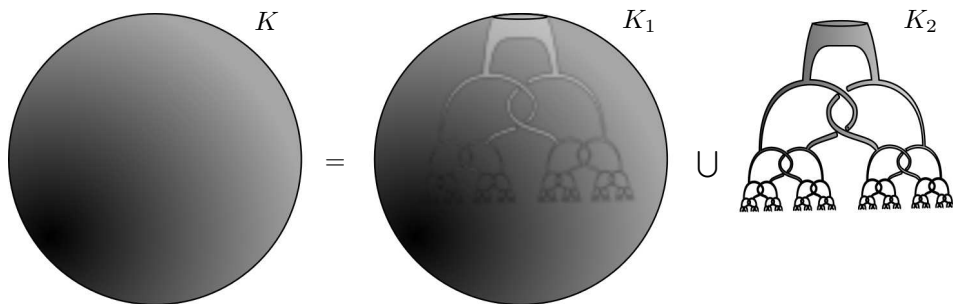}}
	\rput(4.3,1.9){$=$} \rput(9.2,1.9){$\bigcup$}
	\rput(3.4,3.8){$K$} \rput(8.3,3.8){$K_1$} \rput(12,3.8){$K_2$}
\end{pspicture}
\caption{A decomposition of $K$ that is not tame \label{fig:not_tame}}
\end{figure}

The intersection $K_1 \cap K_2$ is precisely ${\rm fr}\ K_2$ (the topological frontier of $K_2$) minus the interior of the disk that lies at its top. It is a consequence of the theorem on invariance of domain that ${\rm fr}\ K_2$ agrees with the boundary of the manifold $K_2$, so ${\rm fr}\ K_2$ is a $2$--sphere and consequently $K_1 \cap K_2$ is a $2$--sphere minus a disk, which is again a disk (this follows from the Sch\"onflies' theorem in the plane).

We have $r(K_1) \geq 0$ by definition, $r(K) = 0$ (this is trivial) and $r(K_2) = \infty$ by Theorem \ref{teo:ball}. Therefore \[\infty = r(K_1) + r(K_2) \not\leq r(K) = 0\]  and so the subadditivity property does not hold for this decomposition.
\end{proof}

\subsection{The sphere of Alexander cannot be an attractor} The boundary of the ball of Alexander $B^*$ is the \emph{sphere of Alexander} $S^*$. We want to show that $S^*$ cannot be an attractor by proving that $r(S^*) = \infty$.

All the work done so far translates verbatim to compact subsets of $\mathbb{S}^3$, rather than $\mathbb{R}^3$. Moreover, if $K \subseteq \mathbb{R}^3$, then $r(K)$ is the same regardless of whether we consider $K$ as a subset of $\mathbb{R}^3$ or $\mathbb{S}^3$. Changing our ambient space from $\mathbb{R}^3$ to $\mathbb{S}^3$ is the most natural setting for what follows.

By a \emph{surface} we understand a compact $2$--manifold without boundary. When $S \subseteq \mathbb{S}^3$ is a \emph{connected} surface, using Alexander duality and Poincar\'e duality we have \[\tilde{H}_0(\mathbb{S}^3-S) = H^2(S) = H_0(S) = \mathbb{Z}_2,\] so $S$ separates $\mathbb{S}^3$ into two complementary domains $U_1$ and $U_2$. The closures $\overline{U}_1, \overline{U}_2$ are compact subsets of $\mathbb{S}^3$, so $r(\overline{U}_1)$ and $r(\overline{U}_2)$ are defined. It turns out that there is a very simple relation among these numbers and $r(S)$:

\begin{proposition} \label{prop:surface} Let $S \subseteq \mathbb{S}^3$ be a connected surface and $U_1$, $U_2$ its complementary domains. Then \[r(S) = r(\overline{U}_1) + r(\overline{U}_2).\]
\end{proposition}

Before proving this result we need Lemma \ref{lem:borde} below. Its proof is an adaptation of an argument contained in the book by Munkres \cite[Theorem 36.3, p. 205]{munkres3}.

\begin{lemma} \label{lem:borde} Let $U_1$ and $U_2$ be the complementary domains of a connected surface $S \subseteq \mathbb{S}^3$. Then their topological frontiers ${\rm fr}\ U_1$ and ${\rm fr}\ U_2$ coincide with $S$.
\end{lemma}
\begin{proof} The only nontrivial part is to show that ${\rm fr}\ U_1$ and ${\rm fr}\ U_2$ are not only subsets of $S$, but that they are all of $S$. To prove this, observe first that the compact set $C := {\rm fr}\ U_1$ separates $\mathbb{S}^3$, for no point $p \in U_1$ can be joined to a point $q \in U_2$ with a path that does not meet $C$. Thus by Spanier's version of Alexander duality \cite[Theorem 17, p. 296]{spanier1} in $\mathbb{S}^3$ we see that $\check{H}^2(C) \neq 0$. Again by Alexander duality, but now in the surface $S$, it follows that $H_0(S,S-C) = \check{H}^2(C) \neq 0$. Since $S$ is connected, this requires that $S - C = \emptyset$, so $C = S$ as we wanted to prove. An analogous argument shows that ${\rm fr}\ U_2 = C$.
\end{proof}

\begin{proof}[Proof of Proposition \ref{prop:surface}] We prove both inequalities.

($\leq$) It will be enough to show that for every neighbourhood $V$ of $S$ there is an $m$--neighbourhood $N$ of $S$ contained in $V$ such that ${\rm rk}\ H_1(N) \leq r(\overline{U}_1) + r(\overline{U}_2)$.

We assume $r(\overline{U}_1), r(\overline{U}_2) < \infty$, for otherwise there is nothing to prove. Since ${\rm fr}\ \overline{U}_1 \subseteq S$ is contained in ${\rm int}\ V$, the set $W_1 := \overline{U}_1 \cup V$ is a neighbourhood of $\overline{U}_1$. Thus there is an $m$--neighbourhood $N_1$ of $\overline{U}_1$ such that $N_1 \subseteq W_1$ and ${\rm rk}\ H_1(N_1) = r(\overline{U}_1)$. Similarly, there is an $m$--neighbourhood $N_2$ of $\overline{U}_2$ such that $N_2 \subseteq W_2 := \overline{U}_2 \cup V$ and ${\rm rk}\ H_1(N_2) = r(\overline{U}_2)$.

Let $N := N_1 \cap N_2$, which is an $m$--neighbourhood of $S$ contained in $W_1 \cap W_2 = V$. Since $\mathbb{S}^3 = N_1 \cup N_2$, there is a Mayer--Vietoris exact sequence \[H_2(\mathbb{S}^3) = 0 \longrightarrow H_1(N) \longrightarrow H_1(N_1) \oplus H_1(N_2),\] whence ${\rm rk}\ H_1(N) \leq {\rm rk}\ H_1(N_1) + {\rm rk}\ H_1(N_2) = r(\overline{U}_1) + r(\overline{U}_2)$.
\smallskip

($\geq$) We assume $r(S) < \infty$, for otherwise there is nothing to prove.

Let $V_1$ and $V_2$ be neighbourhoods of $\overline{U}_1$ and $\overline{U}_2$ respectively. $\overline{U}_1$ and $\overline{U}_2$ both contain $S$ by Lemma \ref{lem:borde}, so $V_1 \cap V_2$ is a neighbourhood of $S$. Thus there exists an $m$--neighbourhood $N$ of $S$ such that $N \subseteq V_1 \cap V_2$ and ${\rm rk}\ H_1(N) = r(S)$. Denote $N_1 := \overline{U}_1 \cup N$ and $N_2 := \overline{U}_2 \cup N$. These are $m$--neighbourhoods of $\overline{U}_1$ and $\overline{U}_2$ respectively. Observe that $N_1 \cap N_2 = N$ and $N_1 \cup N_2 = \mathbb{S}^3$, so from the Mayer--Vietoris sequence \[H_1(N) \longrightarrow H_1(N_1) \oplus H_1(N_2) \longrightarrow H_1(\mathbb{S}^3) = 0\] we see that ${\rm rk}\ H_1(N_1) + {\rm rk}\ H_1(N_2) \leq {\rm rk}\ H_1(N)$. Since $N_1 \subseteq V_1$ and $N_2 \subseteq V_2$, we conclude that $r(\overline{U}_1) + r(\overline{U}_2) \leq r(S)$.
\end{proof}

\begin{theorem} \label{teo:alex} The sphere of Alexander $S^*$ cannot be an attractor.
\end{theorem}
\begin{proof} The closure of one of the complementary domains of $S^*$ is precisely the ball of Alexander $B^*$. From Theorem \ref{teo:ball} we know $r(B^*) = \infty$, and then Proposition \ref{prop:surface} implies $r(S^*) \geq r(B^*) = \infty$.
\end{proof}

\section{Comparing $r(K)$ and \v{C}ech cohomology} \label{sec:cech}

The definition of $r(K)$ suggests that $r(K)$ encodes two different pieces of information about $K$: it captures its ``intrinsic complexity'' as measured by ${\rm rk}\ \check{H}^1(K)$ and, in addition, it includes an extra term that accounts for the ``crookedness'' of $K$ as a subset of $\mathbb{R}^3$. On these intuitive grounds we may expect the inequality $r(K) \geq {\rm rk}\ \check{H}^1(K)$ to hold true, and this is actually the case:

\begin{theorem} \label{teo:bound} Let $K$ be a compact subset of $\mathbb{R}^3$. Then $r(K) \geq {\rm rk}\ \check{H}^1(K)$.
\end{theorem}

The proof of this inequality depends on Lemma \ref{lem:vector} below, which is an easy algebraic result. However, it should be noted that it makes essential use (for the only time in this paper) of the fact that we are computing homologies and cohomologies with coefficients in a field and therefore the homology and cohomology groups are actually vector spaces.  Recall that the dimension of a vector space $V$ agrees with its rank when viewed as an abelian additive group.

\begin{lemma} \label{lem:vector} Let \[V_1 \stackrel{\varphi_1}{\longrightarrow} V_2 \stackrel{\varphi_2}{\longrightarrow} \ldots \stackrel{\varphi_{k-1}}{\longrightarrow} V_k \stackrel{\varphi_k}{\longrightarrow} \ldots\] be a direct sequence of vector spaces and let $V$ be its direct limit. If ${\rm dim}\ V_k \leq r < \infty$ for every $k$, then ${\rm dim}\ V \leq r$.
\end{lemma}
\begin{proof} By the definition of direct limit, each $v \in V$ is of the form $\psi_k(u_k)$ for suitable $k$ and $u_k \in V_k$. It is an easy exercise to see that this holds more generally: if $S \subseteq V$ is finite, then $S = \psi_k(U)$ for suitable $k$ and $U \subseteq V_k$.

Now suppose that ${\rm dim}\ V > r$, so there exists a linearly independent set $S \subseteq V$ with $r+1$ elements. Then $S = \psi_k(U)$ for suitable $k$ and $U \subseteq V_k$, and $U$ must contain at least $r+1$ linearly independent elements. This contradicts the hypothesis that ${\rm dim}\ V_k \leq r$.
\end{proof}

\begin{proof}[Proof of Theorem \ref{teo:bound}] If $r(K) = \infty$ there is nothing to prove, so we assume that $r(K) = r < \infty$. By Proposition \ref{prop:easy1}.{\it ii} $K$ has a $pm$--neighbourhood basis $\{N_k\}$ such that ${\rm rk}\ H_1(N_k) = r$ for every $k$, which implies ${\rm rk}\ H^1(N_k) = r$ by the universal coefficient theorem. We may as well assume that $N_{k+1} \subseteq N_k$ for each $k$, and then by the continuity property of \v{C}ech cohomology $\check{H}^1(K)$ is the direct limit of the sequence \[H^1(N_1) \longrightarrow H^1(N_2) \longrightarrow H^1(N_2) \longrightarrow \ldots\] where the bonding maps are induced by inclusions. Then Lemma \ref{lem:vector} implies that ${\rm rk}\ \check{H}^1(K) \leq r = r(K)$.
\end{proof}

Of course, the inequality in Theorem \ref{teo:bound} may be strict: the nonattracting arcs, balls and spheres constructed earlier all have trivial $\check{H}^1$ and infinite $r$. When ${\rm rk}\ \check{H}^1(K) < \infty$ it will be convenient to denote $c(K)$ the nonnegative integer, possibly $\infty$, such that \[r(K) = {\rm rk}\ \check{H}^1(K) + c(K).\]

We shall accept that $c(K)$ quantifies the ``crookedness'' of $K$ as a subset of $\mathbb{R}^3$ much in the same way as $r(K)$ does, but having factored out the contribution due to the ``intrinsic complexity'' of $K$. Clearly $c$ inherits from $r$ both the invariance and the finiteness properties.
\medskip

It is natural to expect that polyhedra should have $c = 0$. This turns out to be true:

\begin{proposition} \label{prop:rpol} Let $K \subseteq \mathbb{R}^3$ be a polyhedron. Then $c(K) = 0$.
\end{proposition}
\begin{proof} It is well known that ${\rm rk}\ \check{H}^1(K) = {\rm rk}\ H^1(K) < \infty$ for a poyhedron, so $c(K)$ is defined. Let $\{N_k\}$ be a $pm$--neighbourhood basis of $K$ whose members are all regular neighbourhoods of $K$. Each inclusion $K \subseteq N_k$ is a homotopy equivalence so ${\rm rk}\ \check{H}^1(N_k) = {\rm rk}\ \check{H}^1(K)$ for every $k$, which easily implies $r(K) \leq {\rm rk}\ \check{H}^1(K)$. Together with Theorem \ref{teo:bound} this shows that $r(K) = {\rm rk}\ \check{H}^1(K)$ and $c(K) = 0$.
\end{proof}

Another interesting class of sets for which $c = 0$ is that of attractors for \emph{flows}.

\begin{proposition} \label{prop:at_poly} Let $K \subseteq \mathbb{R}^3$ be an attractor for a \emph{flow}. Then $c(K) = 0$.
\end{proposition}
\begin{proof} Since $K$ is also an attractor for the time-one map of the flow, by the finiteness property $r(K) < \infty$ so $c(K)$ is defined. $K$ has arbitrarily small $m$--neighbourhoods $P$ such that the inclusion $K \subseteq P$ induces isomorphisms in Cech cohomology \cite[Proposition 6 and Remark 9, p. 6168]{mio3} which implies $r(K) \leq {\rm rk}\ \check{H}^1(K)$. Together with Theorem \ref{teo:bound} this proves that $c(K) = 0$.
\end{proof}

Proposition \ref{prop:at_poly} can be used to show that certain sets which are attractors for homeomorphisms cannot be attractors for flows. This is the case of the dyadic solenoid and the Whitehead continuum, as we now describe.
\smallskip

{\it The dyadic solenoid.} Let $T \subseteq \mathbb{R}^3$ be a solid torus, and let $f : \mathbb{R}^3 \longrightarrow \mathbb{R}^3$ be a homeomorphism such that $f(T)$ winds twice inside $T$ as shown in Figure \ref{fig:solenoid}. The set $K := \bigcap_k f^k(T)$ is known as the \emph{dyadic solenoid}, and it is an attractor for $f$ by its very construction.

\begin{figure}[h]
\begin{pspicture}(0,0)(8.4,3)
%\psgrid(0,0)(8.4,3)
	\rput[bl](0,0){\scalebox{0.7}{\includegraphics{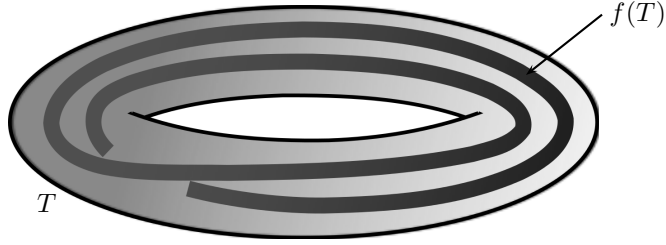}}}
	\rput(0.5,0.5){$T$}
	\psline{->}(7.8,3)(6.8,2.2) \rput[l](7.9,3){$f(T)$}
\end{pspicture}
\caption{Constructing the dyadic solenoid \label{fig:solenoid}}
\end{figure}

\begin{example} \label{ejem:solenoid} The dyadic solenoid $K$ has $c(K) = 1$. Thus by Proposition \ref{prop:at_poly} it cannot be an attractor for a flow, although it is an attractor for a homeomorphism.
\end{example}
\begin{proof} For each $k$ let $N_k$ be the solid torus $f_k(T)$. Clearly $\{N_k\}$ is an $m$--neighbourhood basis of $K$, so $r(K) \leq 1$ because ${\rm rk}\ H_1(N_k) = 1$ for every $k$. Using the continuity property of \v{C}ech cohomology it is very easy to see that $\check{H}^d(K;\mathbb{Z}_2) = 0$ for every $d \geq 1$ but $\check{H}^1(K;\mathbb{Z}) \neq 0$. The latter implies that $K$ is not cellular, and then by the nullity property $r(K) \geq 1$. It follows that $r(K) = 1$ and $c(K) = 1$ too.
\end{proof}

{\it The Whitehead continuum.} Let $T \subseteq \mathbb{R}^3$ be a solid torus, and let $f : \mathbb{R}^3 \longrightarrow \mathbb{R}^3$ be a homeomorphism such that $f(T)$ lies inside $T$ as shown in Figure \ref{fig:wh}. The set $K := \bigcap_k f^k(T)$ is known as the \emph{Whitehead continuum}, and it is an attractor for $f$ by its very construction.

\begin{figure}[h]
\begin{pspicture}(0,0)(8.4,3)
%\psgrid(0,0)(8.4,3)
	\rput[bl](0,0){\scalebox{0.7}{\includegraphics{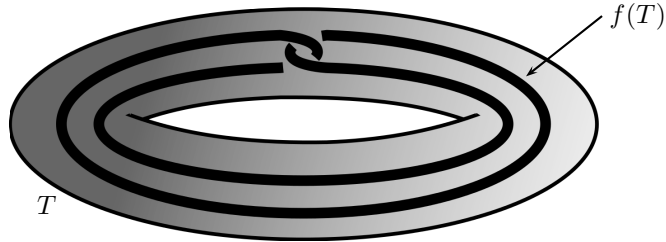}}}
	\rput(0.5,0.5){$T$}
	\psline{->}(7.8,3)(6.8,2.2) \rput[l](7.9,3){$f(T)$}
\end{pspicture}
\caption{Constructing the Whitehead continuum \label{fig:wh}}
\end{figure}

\begin{example} \label{ejem:wh} The Whitehead continuum $K$ has $c(K) = 1$. Thus by Proposition \ref{prop:at_poly} it cannot be an attractor for a flow, although it is an attractor for a homeomorphism.
\end{example}
\begin{proof} As in Example \ref{ejem:solenoid}, $r(K) \leq 1$. Each torus $f^{k+1}(T)$ is nullhomotopic inside the previous one $f^k(T)$, so $\check{H}^d(K;\mathbb{Z}) = 0$ for every $d \geq 1$. The Whitehead continuum is not cellular \cite[p. 156]{hempel1}, so by the nullity property $r(K) = 1$. Thus $c(K) = 1$.
\end{proof}

If the reader is familiar with shape theory \cite{borsukshape1} he may want to consider the following remark. There is a result of G\"unther and Segal \cite[Corollary 1, p. 324]{gunthersegal1} which states that every attractor for a flow on a manifold has the shape of a (finite) polyhedron, so in particular it has finitely generated \v{C}ech cohomology in every dimension and for every coefficient group (this holds true even if attractors are allowed to be mildly unstable \cite{moronsanjurjoyo1,mio1}). Since the dyadic solenoid has nonfinitely generated $\check{H}^1(K;\mathbb{Z})$, it follows that it cannot be an attractor for a flow \cite[Example 3, p. 325]{gunthersegal1} which agrees with Example \ref{ejem:solenoid}. Conversely, G\"unther and Segal also show that a (finite dimensional, metrizable) compact set having the shape of a finite polyhedron can be embedded in some Euclidean space in such a way that it is an attractor for a suitable flow \cite[Theorem 2, p. 327]{gunthersegal1}. The Whitehead continuum \emph{does} have the shape of a finite polyhedron (it has the shape of a point), but according to Example \ref{ejem:wh} it cannot be an attractor in $\mathbb{R}^3$. However, if $K$ is thought of as a subset of $\mathbb{R}^4$ (identifying $\mathbb{R}^3$ with $\mathbb{R}^3 \times \{0\} \subseteq \mathbb{R}^4$) then it becomes cellular and it \emph{is} an attractor for a suitable flow in $\mathbb{R}^4$. This shows how shape theory, which is a topological invariant and cannot discriminate between $K \subseteq \mathbb{R}^3$ and $K \subseteq \mathbb{R}^4$, cannot be used to analyse Example \ref{ejem:wh}.

\section{Final remarks and open questions} \label{sec:final}

We finish this paper by stating some open questions.

\subsection{} Of course,

\smallskip
{\bf Question 1.} Solve the characterization problem for $M = \mathbb{R}^3$.
\smallskip

\noindent or even more ambitiously, solve the characterization problem when the dynamics on the attractor is given:

\smallskip
{\bf Question 2.} Assume $K \subseteq \mathbb{R}^3$ is a compact set on which a homeomorphism $h : K \longrightarrow K$ is defined. Characterize when it is possible to extend $h$ to a homeomorphism $\hat{h} : \mathbb{R}^3 \longrightarrow \mathbb{R}^3$ having $K$ as an attractor.
\smallskip

There is a very interesting situation, related to the theory of partial differential equations (PDEs), in which this question arises naturally. A PDE generates a flow or semiflow on an infinite-dimensional phase space, but many of them (especially those with a physical motivation) have finite-dimensional attractors. A copy $K$ of such an attractor can be embedded together with its dynamics in some finite dimensional Euclidean space $E$, and the question arises whether there exists a flow in $E$ having $K$ as an attractor and reproducing its dynamics. This can be considered a ``parametric version'' of the extension problem posed above (with $\mathbb{R}^3$ replaced by $E$). Heuristically, if such a flow exists then a finite number of variables suffices to describe completely the long term behaviour of the original system modelled by the PDE. The question was first studied by Eden, Foias, Nicolaenko and Temam \cite{eden1} as well as Romanov \cite{romanov1} but only later \cite{pintorobinsonyo1, robinson1, robinsonyo1} it was noticed that the way $K$ is embedded in $E$ needs to be carefully controlled; it also became clear that the mathematical language needed to describe how $K$ should sit in $E$ is still to be developed. We hope that the ideas introduced in the present paper provide a starting point for this task.

\subsection{} Suppose $K$ is an attractor for $f$. Let $N \subseteq \mathcal{A}(K)$ be an $m$--neighbourhood of $K$. Since $K$ attracts $N$, there exists $n_0 \in \mathbb{N}$ such that $f^{n_0}(N) \subseteq {\rm int}\ N$. Letting $N_k := f^{kn_0}(N)$ we obtain (as in Lemma \ref{lem:trivial}) a \emph{decreasing} $m$--neighbourhood basis of $K$ all of whose members are homeomorphic to each other, but with the additional property that $f^k$ provides a homeomorphism from the pair $(N_0,N_1)$ onto the pair $(N_k,N_{k+1})$. Less technically stated, $\{N_k\}$ has the property that each $N_{k+1}$ lies in $N_k$ in the same way as $N_1$ lies in $N_0$.

\smallskip
{\bf Question 3.} Is it possible to refine the definition of $r(K)$ in such a way that it accounts for the above observation?
\smallskip

Let us include a specific example to illustrate this question. Denote $\nu_k$ the sequence of prime numbers $2,3,5,7,11,13,\ldots$ Consider a solid torus $T_0 \subseteq \mathbb{R}^3$. Place in its interior another solid torus $T_1$ that winds $\nu_1 = 2$ times around $T_0$. Then place in the interior of $T_1$ another solid torus $T_2$ that winds $\nu_2 = 3$ times around $T_1$, and so on. Repeating this construction one obtains a nested sequence of tori $\{T_k\}$, each winding $\nu_k$ times around the previous one, whose intersection is a compact set $K$. Clearly $r(K) \leq 1 < \infty$, so $r$ does not rule out the possibility that $K$ is an attractor, but our intuition tells us that it cannot because there is no repeating pattern in how each $T_{k+1}$ lies in the previous $T_k$. A beautiful result of G\"unther \cite[Theorem 1, p. 653]{gunther1} confirms this intuition by an ingenious examination of the structure of $\check{H}^1(K;\mathbb{Z})$.

\subsection{} Throughout the paper we have concentrated on discrete dynamical systems, but the characterization problem makes perfect sense also in the context of continuous dynamical systems: ``characterize topologically what compact sets $K \subseteq \mathbb{R}^3$ can be attractors for \emph{flows} on $\mathbb{R}^3$''. Being an attractor for a flow is much more restrictive than being an attractor for a homeomorphism (Examples \ref{ejem:solenoid} and \ref{ejem:wh} illustrate this), which makes this version of the characterization problem easier to deal with. It can be answered in the following terms:
\smallskip

{\it Theorem.}  \cite[Theorem 11, p. 6169]{mio3} A compactum $K \subseteq \mathbb{R}^3$ is an attractor for a flow if and only if there exists a polyhedron $P \subseteq \mathbb{R}^3$ such that $\mathbb{S}^3 - K$ is homeomorphic to $\mathbb{S}^3 - P$.
\smallskip

The interesting point to be observed here is that whether $K$ is an attractor for a flow depends \emph{only} on $\mathbb{S}^3 - K$. 

\smallskip
{\bf Question 4.} Suppose $K, K' \subseteq \mathbb{R}^3$ are compacta such that $\mathbb{S}^3 - K$ and $\mathbb{S}^3 - K'$ are homeomorphic. Is it true that if $K$ is an attractor for a homeomorphism, then so must be $K'$?
\smallskip

If the answer to this question is negative (that is, if there exist compacta with homeomorphic complements in $\mathbb{S}^3$, one being an attractor and the other one not) then the characterization problem cannot be solved in terms of invariants such as $r(K)$, for we saw earlier that these depend \emph{only} on $\mathbb{S}^3 - K$ (Theorem \ref{teo:invs}). 

\subsection{} If the reader is acquainted with the notions of \emph{wild} and \emph{tame} sets from geometric topology he might have been reminded of them by the overall flavour of the paper. Recall that a compact set $K \subseteq \mathbb{R}^3$ is called \emph{tame} if there exist a polyhedron $P \subseteq \mathbb{R}^3$ and a homeomorphism $h : \mathbb{R}^3 \longrightarrow \mathbb{R}^3$ such that $h(K) = P$; otherwise $K$ is said to be \emph{wild}. It is almost unavoidable to wonder whether there is some relation between these concepts and the characterization problem. One has the following result:

\smallskip
{\it Proposition.} \cite[Proposition 12, p. 6169]{mio3} Every tame set is an attractor for a flow; hence also for a homeomorphism (namely, the time-one map of the flow).
\smallskip

This was essentially proved by G\"unther and Segal \cite[Corollary 4, p. 327]{gunthersegal1}. It explains why all our examples of nonattracting sets are wild, although one should not be misled to think that no wild set can be an attractor: for instance, the dyadic solenoid is wild but it is an attractor, and it is even possible to construct wild arcs that are attractors \cite[Example 38, p. 6177]{mio3}.
\smallskip

The notions of tameness and wildness have local versions when applied to arcs or spheres. An arc $A \subseteq \mathbb{R}^3$ is \emph{locally tame} at a point $p \in A$ if there exist a neighbourhood $U$ of $p$ in $\mathbb{R}^3$ and a homeomorphism $h : \mathbb{R}^3 \longrightarrow \mathbb{R}^3$ such that $h(U \cap A)$ is contained in the $x$--axis. Similarly, a sphere $S \subseteq \mathbb{R}^3$ is locally tame at a point $p \in S$ if there exist a neighbourhood $U$ of $p$ in $\mathbb{R}^3$ and a homeomorphism $h : \mathbb{R}^3 \longrightarrow \mathbb{R}^3$ such that $h(U \cap S)$ is contained in the $xy$--plane. In both cases, the set of points at which the given set $K$ is tame is an open subset of $K$; its complement is a compact subset of $K$ called the \emph{wild} set of $K$.

\smallskip
{\bf Question 5.} Is it possible to characterize when an arc or a sphere is an attractor in terms of properties of their wild sets?
\smallskip

\subsection{} We used the Fox--Artin arc $F$ shown in Figure \ref{fig:wild_2} as a basis for our construction of an arc that cannot be an attractor. It is fairly easy to construct a neighbourhood basis of $F$ comprised of double tori, so $r(F) \leq 2$. We also know that $r(F) \geq 1$ because $F$ is not cellular.

\smallskip
{\bf Question 6.} What is $r(F)$? Can $F$ be an attractor?
\smallskip

\section*{Appendix: Alexander duality for manifolds with boundary}

\begin{lemma} Let $N$ be a compact $3$--manifold and $S \subseteq {\rm int}\ N$ a compact set. Then $H_d(N,S) = H_{3-d}(N-S,\partial N)$.
\end{lemma}
\begin{proof} $\partial N$ is collared in $N$ \cite{connelly1}. This means that there exists an embedding $c : \partial N \times [0,1] \longrightarrow N$ such that $c(p,0) = p$ for every $p \in \partial N$. Since $S$ is a compact subset of ${\rm int}\ N$, there exists $\varepsilon > 0$ such that $c(\partial N \times [0,\varepsilon])$ is disjoint from $S$. The theorem on invariance of domain guarantees that $U := c(\partial N \times [0,\varepsilon))$ is an open subset of $N$, and it is not difficult to check that $N^* := N - U$ is a compact $3$--manifold contained in ${\rm int}\ N$ and containing $S$ in its interior. The inclusions $N^* \subseteq N$ and $\partial N \subseteq N - N^*$ are both homotopy equivalences.

The cohomology exact sequence for the triple $(N,N^*,S)$ and the fact that $\check{H}^*(N,N^*) = 0$ imply that $\check{H}^d(N,S) = \check{H}^d(N^*,S)$. Using Spanier's version of Alexander duality \cite[Theorem 17, p. 296]{spanier1} for the compact pair $(N^*,S)$ in the interior of the boundariless manifold ${\rm int}\ N$ we have \[\check{H}^d(N^*,S) = H_{3-d}(({\rm int}\ N) - S, ({\rm int}\ N) - N^*).\] An easy exercise involving the collar $c$ shows that $H_{3-d}(({\rm int}\ N) - S,({\rm int}\ N) - N^*) = H_{3-d}(N-S,N-N^*)$ so \[\check{H}^d(N^*,S) = H_{3-d}(N-S,N-N^*),\] which readily implies that \[\check{H}^d(N,S) = H_{3-d}(N-S,\partial N).\]
\end{proof}

\bibliographystyle{plain}
\bibliography{biblio}

\end{document}